\documentclass[a4paper]{amsart}
\usepackage{amsmath,amsthm,amssymb,latexsym,epic,bbm,comment}
\usepackage{graphicx,enumerate,stmaryrd,color,todonotes}
\usepackage[all]{xy}
\xyoption{2cell}

\usepackage[active]{srcltx}
\usepackage[parfill]{parskip}

\newtheorem{theorem}{Theorem}[subsection]
\newtheorem{lemma}[theorem]{Lemma}
\newtheorem{corollary}[theorem]{Corollary}
\newtheorem{proposition}[theorem]{Proposition}
\theoremstyle{definition}
\newtheorem{remark}[theorem]{Remark}
\newtheorem{example}[theorem]{Example}

\newtheorem{assum}[theorem]{Assumption}

\newcommand{\tto}{\twoheadrightarrow}

\font\sc=rsfs10
\newcommand{\cC}{\sc\mbox{C}\hspace{1.0pt}}

\newcommand{\cJ}{\sc\mbox{J}\hspace{1.0pt}}
\newcommand{\cS}{\sc\mbox{S}\hspace{1.0pt}}

\newcommand{\cA}{\sc\mbox{A}\hspace{1.0pt}}
\newcommand{\cB}{\sc\mbox{B}\hspace{1.0pt}}

\font\scc=rsfs7
\newcommand{\ccC}{\scc\mbox{C}\hspace{1.0pt}}

\newcommand{\ccA}{\scc\mbox{A}\hspace{1.0pt}}

\newcommand{\ccS}{\scc\mbox{S}\hspace{1.0pt}}

\newcommand{\imorphism}{\spadesuit}

\begin{document}

\title[Endomorphisms of cell $2$-representations]{Endomorphisms of cell $2$-representations}
\author{Volodymyr Mazorchuk and Vanessa Miemietz}

\begin{abstract}
We determine the endomorphism categories of cell $2$-re\-pre\-sen\-tations of fiat $2$-categories
associated with strongly regular two-sided cells and classify, up to biequivalence, 
$\mathcal{J}$-simple fiat $2$-categories which have only one two-sided cell $\mathcal{J}$ 
apart from the identities.
\end{abstract}

\maketitle

\section{Introduction and description of the results}\label{s0}

Classically, Schur's Lemma asserts that the endomorphism algebra of a simple module (say for a 
finite dimensional algebra $A$ over some algebraically closed field $\Bbbk$) is isomorphic to $\Bbbk$.
It might happen that the algebra $A$ is obtained by decategorifying some $2$-category and that the simple
module in question is the decategorification of some $2$-representation of $A$. It is then natural
to ask whether the assertion of Schur's Lemma is the $1$-shadow of some $2$-analogue. Put differently,
this is a question about the endomorphism category of a $2$-representation of some $2$-category.

In \cite{MM} we defined a class of $2$-categories, which we call fiat $2$-categories, forming a natural $2$-analogue 
of finite dimensional cellular algebras. Examples of fiat $2$-categories appear (sometimes in disguise) in e.g.
\cite{BG,CR,FKS,KL,La,Ro}. Fiat $2$-categories have certain $2$-representations called {\em cell $2$-representations}, 
which were also defined in \cite{MM}. These $2$-representations satisfy some natural generalizations of the concept 
of simplicity for representations of finite dimensional algebras. The main objective of the present paper 
is to study the endomorphism categories of these cell $2$-representations with the ultimate goal to 
establish a $2$-analogue of Schur's Lemma.

We start the paper by extending the $2$-setup from \cite{MM} to accommodate non-strict $2$-natural transformations
between $2$-representations of fiat $2$-categories. This is done in Section~\ref{s1}, which also contains all
necessary preliminaries. The advantage of our new setup is the fact that $2$-natural transformations become closed
under isomorphism of functors and under taking inverses of equivalences (see Subsection~\ref{s1.33}).

Cell $2$-representations of fiat $2$-categories have particularly nice properties for so-called strongly regular
cells, see Subsection~\ref{s1.5}. In particular, the main result of \cite{MM4} asserts that in this cases
cell $2$-representations exhaust all simple transitive $2$-representations. This is the main case of our
study in this paper. Our main result is that the endomorphism category of such a cell $2$-representation is 
equivalent to  $\Bbbk\text{-}\mathrm{mod}$, see  Theorem~\ref{thm15} in  Section~\ref{s4}.

Along the way, we prove two further interesting results. Firstly, we establish $2$-fullness
for cell $2$-representations with respect to the class of $1$-morphisms in the two-sided cell, see 
Corollary~\ref{cor8} in Subsection~\ref{s3.4}. Secondly, we completely describe fiat $2$-categories which have 
only one two-sided cell $\mathcal{J}$ apart from the identities, in the case when our $2$-category is 
$\mathcal{J}$-simple  in the sense of \cite{MM2}, see Theorem~\ref{thm11} in Subsection~\ref{s3.6}.
This can be viewed as a $2$-analogue of Artin-Wedderburn Theorem.

We present various examples in Section~\ref{s6}, including the fiat $2$-category of Soergel bimodules acting on the
principal block of the BGG category $\mathcal{O}$ and the fiat $2$-category associated with the
$\mathfrak{sl}_2$-categorification of Chuang and Rouquier. Finally, in Section~\ref{s7}, we introduce 
and investigate a natural setup for the study of graded fiat $2$-categories.
\vspace{5mm}

{\bf Remark.} The original version of the paper appeared on arxiv in July 2012. The present version 
is a substantial revision of the original one which takes into account that since the publication of the
original version several results and assumptions became obsolete due to further developments presented
in \cite{MM3,MM4,MM5}. 
\vspace{5mm}

\noindent
{\bf Acknowledgment.} A substantial part of the paper was written during a visit of the second author to 
Uppsala University, whose  hospitality is gratefully acknowledged. The visit was supported by the Swedish 
Research Council and the Department of Mathematics. The first author is partially supported by the 
Swedish Research Council and the Royal Swedish Academy of Sciences. The second author is partially 
supported by ERC grant PERG07-GA-2010-268109 and EPSRC grant EP/K011782/1.

\section{Preliminaries}\label{s1}

We denote by $\mathbb{N}$ and $\mathbb{N}_0$ the sets of positive and non-negative integers, respectively.

\subsection{Various $2$-categories}\label{s1.1}

In this paper by a $2$-category we mean a strict locally small $2$-category (see \cite{Le} for a concise 
introduction to $2$-categories and bicategories). Let $\cC$ be a $2$-category. 
We will use $\mathtt{i},\mathtt{j},\dots$  to denote objects in $\cC$;
$1$-morphisms in $\cC$ will be denoted by $\mathrm{F},\mathrm{G},\dots$; $2$-morphisms in $\cC$ will be denoted 
by $\alpha,\beta,\dots$. For $\mathtt{i}\in\cC$ we will denote by $\mathbbm{1}_{\mathtt{i}}$ the corresponding
identity $1$-morphisms. For a $1$-morphism $\mathrm{F}$ we will denote by $\mathrm{id}_{\mathrm{F}}$ the 
corresponding identity $2$-morphisms. 

Denote by $\mathbf{Cat}$ the $2$-category of all small categories. Let $\Bbbk$ be an algebraically closed field. 
Denote by $\mathfrak{A}_{\Bbbk}$ the $2$-category whose objects are small $\Bbbk$-linear fully additive categories;
$1$-morphisms are additive $\Bbbk$-linear functors and $2$-morphisms are natural transformations. Denote by 
$\mathfrak{A}_{\Bbbk}^f$ the full $2$-subcategory of $\mathfrak{A}_{\Bbbk}$ whose objects are 
fully additive categories $\mathcal{A}$ such that $\mathcal{A}$ has only finitely many 
isomorphism classes of indecomposable objects and all morphisms spaces in $\mathcal{A}$ are finite dimensional.
We also denote by $\mathfrak{R}_{\Bbbk}$ the full subcategory of $\mathfrak{A}_{\Bbbk}$ containing all
objects which are equivalent to $A\text{-}\mathrm{mod}$ for some finite dimensional associative 
$\Bbbk$-algebra $A$.

\subsection{Finitary and fiat $2$-categories}\label{s1.2}

A $2$-category $\cC$ is called {\em finitary (over $\Bbbk$)}, see  \cite{MM}, if the following conditions are satisfied:
\begin{itemize}
\item $\cC$ has finitely many objects;
\item for any $\mathtt{i},\mathtt{j}\in\cC$ we have $\cC(\mathtt{i},\mathtt{j})\in \mathfrak{A}_{\Bbbk}^f$
and horizontal composition is both additive and $\Bbbk$-linear;
\item for any $\mathtt{i}\in\cC$ the $1$-morphism $\mathbbm{1}_{\mathtt{i}}$ is indecomposable.
\end{itemize}
We will call  $\cC$  {\em weakly fiat} provided that it has a weak object preserving anti-autoequivalence $*$
and for any $1$-morphism $\mathrm{F}\in\cC(\mathtt{i},\mathtt{j})$ there exist $2$-morphisms $\alpha:\mathrm{F}\circ\mathrm{F}^*\to
\mathbbm{1}_{\mathtt{j}}$ and $\beta:\mathbbm{1}_{\mathtt{i}}\to \mathrm{F}^*\circ\mathrm{F}$ such that 
$\alpha_{\mathrm{F}}\circ_1\mathrm{F}(\beta)=\mathrm{id}_{\mathrm{F}}$ and
$\mathrm{F}^*(\alpha)\circ_1\beta_{\mathrm{F}^*}=\mathrm{id}_{\mathrm{F}^*}$.
If $*$ is involutive, then $\cC$ is called  {\em fiat}, see \cite{MM}.

\subsection{$2$-representations}\label{s1.3}

From now on $\cC$ will denote a finitary $2$-category. By a $2$-re\-pre\-sen\-ta\-tion of $\cC$ 
we mean a strict $2$-functor 
from $\cC$ to either $\mathfrak{A}_{\Bbbk}$ (additive $2$-re\-pre\-sen\-tation), $\mathfrak{A}_{\Bbbk}^f$ (finitary 
$2$-representation), or $\mathfrak{R}_{\Bbbk}$ (abelian $2$-representation). In this paper we define the
$2$-categories of $2$-representations of $\cC$ extending the setup (from the one in \cite{MM,MM2}) by considering
non-strict $2$-natural transformations between two $2$-representations $\mathbf{M}$ and $\mathbf{N}$.
Such a $2$-natural transformation $\Psi$ consists of the following data:
a map, which assigns to every $\mathtt{i}\in\cC$ a functor $\Psi_{\mathtt{i}}:\mathbf{M}(\mathtt{i})\to 
\mathbf{N}(\mathtt{i})$, and for any $1$-morphism $\mathrm{F}\in\cC(\mathtt{i},\mathtt{j})$ 
a natural isomorphism $\eta_{\mathrm{F}}=\eta^{\Psi}_{\mathrm{F}}:\Psi_{\mathtt{j}}\circ \mathbf{M}(\mathrm{F})\to 
\mathbf{N}(\mathrm{F})\circ \Psi_{\mathtt{i}}$, where naturality means that for any 
$\mathrm{G}\in\cC(\mathtt{i},\mathtt{j})$ and any $\alpha:\mathrm{F}\to \mathrm{G}$
we have 
\begin{displaymath}
\eta_{\mathrm{G}}\circ_1 (\mathrm{id}_{\Psi_{\mathtt{j}}}\circ_0 \mathbf{M}(\alpha))=
(\mathbf{N}(\alpha)\circ_0 \mathrm{id}_{\Psi_{\mathtt{i}}})\circ_1\eta_{\mathrm{F}}. 
\end{displaymath}
In other words, the left diagram on the following picture commutes up to $\eta_{\mathrm{F}}$ while the right
diagram commutes (compare with \cite[Subsection~2.2]{Kh}):
\begin{displaymath}
\xymatrix@R=3mm{ 
\mathbf{M}(\mathtt{i})\ar[rrr]^{\mathbf{M}(\mathrm{F})}\ar[ddd]_{\Psi_{\mathtt{i}}}&&&
\mathbf{M}(\mathtt{j})\ar[ddd]^{\Psi_{\mathtt{j}}}\\&&&\\&\ar@{<:}[ru]^{\eta_{\mathrm{F}}}&\\
\mathbf{N}(\mathtt{i})\ar[rrr]^{\mathbf{N}(\mathrm{F})}
&&&\mathbf{N}(\mathtt{j})\\
}
\quad
\xymatrix{ 
\Psi_{\mathtt{j}}\circ \mathbf{M}(\mathrm{F})\ar[rr]^{\eta_{\mathrm{F}}}
\ar[d]_{\mathrm{id}_{\Psi_{\mathtt{j}}}\circ_0 \mathbf{M}(\alpha)}
&& \mathbf{N}(\mathrm{F})\circ \Psi_{\mathtt{i}}
\ar[d]^{\mathbf{N}(\alpha)\circ_0 \mathrm{id}_{\Psi_{\mathtt{i}}}}\\
\Psi_{\mathtt{j}}\circ \mathbf{M}(\mathrm{G})\ar[rr]^{\eta_{\mathrm{G}}}
&& \mathbf{N}(\mathrm{G})\circ \Psi_{\mathtt{i}}
}
\end{displaymath}
Moreover, the isomorphisms $\eta$ should satisfy 
\begin{equation}\label{eq3}
\eta_{\mathrm{F}\circ_0\mathrm{G}}=(\mathrm{id}_{\mathbf{N}(\mathrm{F})}\circ_0\eta_{\mathrm{G}})\circ_1
(\eta_{\mathrm{F}}\circ_0\mathrm{id}_{\mathbf{M}(\mathrm{G})})
\end{equation}
for all composable $1$-morphisms $\mathrm{F}$ and $\mathrm{G}$.

Given two $2$-natural transformations $\Psi$ and $\Phi$ as above, a modification $\theta:\Psi\to\Phi$ is a map
which assigns to each $\mathtt{i}\in\cC$ a natural transformation $\theta_{\mathtt{i}}:\Psi_{\mathtt{i}}\to
\Phi_{\mathtt{i}}$ such that for any $\mathrm{F},\mathrm{G}\in\cC(\mathtt{i},\mathtt{j})$ and any 
$\alpha:\mathrm{F}\to \mathrm{G}$ we have 
\begin{equation}\label{eq33}
\eta_{\mathrm{G}}^{\Phi}\circ_1 (\theta_{\mathtt{j}}\circ_0 \mathbf{M}(\alpha))=
(\mathbf{N}(\alpha)\circ_0 \theta_{\mathtt{i}})\circ_1\eta_{\mathrm{F}}^{\Psi}. 
\end{equation}

\begin{proposition}\label{prop61}
Together with non-strict $2$-natural transformations and modifications as 
defined above, $2$-representations of $\cC$ form a $2$-category. 
\end{proposition}

Our notation for these $2$-categories is $\cC\text{-}\mathrm{amod}$ in the case of additive representations 
and  $\cC\text{-}\mathrm{afmod}$ in the case of finitary representations. To define the
$2$-category $\cC\text{-}\mathrm{mod}$ for abelian representations we additionally assume that
all $\Psi_{\mathtt{i}}$ are right exact (this assumption is missing in \cite{MM}). 

\begin{proof}
To check that these are $2$-categories, we have to  verify that (strict) composition of non-strict 
$2$-natural transformations is a non-strict $2$-natural transformation and that both horizontal and vertical 
compositions of modifications are modifications.  The first fact follows by defining 
\begin{displaymath}
\eta_{\mathrm{F}}^{\Psi'\circ\Psi}:=(\eta_{\mathrm{F}}^{\Psi'}\circ_0 \mathrm{id}_{\Psi_{\mathtt{i}}})\circ_1
(\mathrm{id}_{\Psi'_{\mathtt{j}}}\circ_0\eta_{\mathrm{F}}^{\Psi})
\end{displaymath}
and then checking \eqref{eq3} (which is a straightforward computation). Since the diagrams
\begin{displaymath}
\xymatrix@C=6mm{
\Psi'_{\mathtt{j}}\circ\Psi_{\mathtt{j}}\circ\mathbf{M}(\mathrm{F})
\ar[rrr]^{\mathrm{id}_{\Psi'_{\mathtt{j}}}\circ_0 \theta_{\mathtt{j}}\circ_0 \mathrm{id}_{\mathbf{M}(\mathrm{F})}}
\ar[d]_{\mathrm{id}_{\Psi'_{\mathtt{j}}}\circ_0\eta^{\Psi}_{\mathrm{F}}}
&&&\Psi'_{\mathtt{j}}\circ\Phi_{\mathtt{j}}\circ\mathbf{M}(\mathrm{F})
\ar[rrr]^{\theta'_{\mathtt{j}}\circ_0 \mathrm{id}_{\Phi_{\mathtt{j}}}\circ_0 \mathrm{id}_{\mathbf{M}(\mathrm{F})}}
\ar[d]|-{\mathrm{id}_{\Psi'_{\mathtt{j}}}\circ_0\eta^{\Phi}_{\mathrm{F}}}
&&&\Phi'_{\mathtt{j}}\circ\Phi_{\mathtt{j}}\circ\mathbf{M}(\mathrm{F})
\ar[d]^{\mathrm{id}_{\Phi'_{\mathtt{j}}}\circ_0\eta^{\Phi}_{\mathrm{F}}}\\
\Psi'_{\mathtt{j}}\circ\mathbf{N}(\mathrm{F})\circ\Psi_{\mathtt{i}}
\ar[rrr]^{\mathrm{id}_{\Psi'_{\mathtt{j}}}\circ_0 \mathrm{id}_{\mathbf{N}(\mathrm{F})}\circ_0\theta_{\mathtt{i}}}
\ar[d]_{\eta^{\Psi'}_{\mathrm{F}}\circ_0\mathrm{id}_{\Psi_{\mathtt{i}}}}
&&&\Psi'_{\mathtt{j}}\circ\mathbf{N}(\mathrm{F})\circ\Phi_{\mathtt{i}}
\ar[rrr]^{\theta'_{\mathtt{j}}\circ_0 \mathrm{id}_{\mathbf{N}(\mathrm{F})}\circ_0 \mathrm{id}_{\Phi_{\mathtt{i}}}}
\ar[d]|-{\eta^{\Psi'}_{\mathrm{F}}\circ_0\mathrm{id}_{\Phi_{\mathtt{i}}}}
&&&\Phi'_{\mathtt{j}}\circ\mathbf{N}(\mathrm{F})\circ\Phi_{\mathtt{i}}
\ar[d]^{\eta^{\Phi'}_{\mathrm{F}}\circ_0\mathrm{id}_{\Phi_{\mathtt{i}}}}\\
\mathbf{K}(\mathrm{F})\circ\Psi'_{\mathtt{i}}\circ\Psi_{\mathtt{i}}
\ar[rrr]^{\mathrm{id}_{\mathbf{K}(\mathrm{F})}\circ_0 \mathrm{id}_{\Psi'_{\mathtt{i}}}\circ_0\theta_{\mathtt{i}} }
&&&\mathbf{K}(\mathrm{F})\circ\Psi'_{\mathtt{i}}\circ\Phi_{\mathtt{i}}
\ar[rrr]^{\mathrm{id}_{\mathbf{K}(\mathrm{F})}\circ_0 \theta'_{\mathtt{i}}\circ_0 \mathrm{id}_{\Phi_{\mathtt{i}}}}
&&&\mathbf{K}(\mathrm{F})\circ\Phi'_{\mathtt{i}}\circ\Phi_{\mathtt{i}}
}
\end{displaymath}
\hspace{2mm}

\begin{displaymath}
\xymatrix{
\Psi_{\mathtt{j}}\circ\mathbf{M}(\mathrm{F})\ar[rr]^{\theta_{\mathtt{j}}\circ_0 \mathbf{M}(\alpha)}
\ar[d]_{\eta_{\mathrm{F}}^{\Psi}}&& 
\Phi_{\mathtt{j}}\circ\mathbf{M}(\mathrm{G})\ar[rr]^{\tau_{\mathtt{j}}\circ_0 \mathrm{id}_{\mathbf{M}(\mathrm{G})}}
\ar[d]_{\eta_{\mathrm{G}}^{\Phi}}
&&\Sigma_{\mathtt{j}}\circ\mathbf{M}(\mathrm{G})\ar[d]^{\eta_{\mathrm{G}}^{\Sigma}}\\
\mathbf{N}(\mathrm{F})\circ\Psi_{\mathtt{i}}\ar[rr]^{\mathbf{N}(\alpha)\circ_0 \theta_{\mathtt{i}}}&&
\mathbf{N}(\mathrm{G})\circ\Phi_{\mathtt{i}}\ar[rr]^{\mathrm{id}_{\mathbf{N}(\mathrm{G})}\circ_0 \tau_{\mathtt{i}}}&&
\mathbf{N}(\mathrm{G})\circ\Sigma_{\mathtt{i}}
}
\end{displaymath}
commute, the latter two facts also follow.
\end{proof}

\subsection{Properties of $2$-natural transformations}\label{s1.33}

Let $\mathbf{M}$ and $\mathbf{N}$ be two $2$-representations of $\cC$ and $\Psi:\mathbf{M}\to \mathbf{N}$ a 
$2$-natural transformation. Given, for every $\mathtt{i}\in \cC$, a functor $\Phi_{\mathtt{i}}$ and an
isomorphism $\xi_{\mathtt{i}}:\Phi_{\mathtt{i}}\to \Psi_{\mathtt{i}}$, define, for every $1$-morphism
$\mathrm{F}\in \cC(\mathtt{i},\mathtt{j})$
\begin{displaymath}
\eta_{\mathrm{F}}^{\Phi}:=(\mathrm{id}_{\mathbf{N}(\mathrm{F})}\circ_0\xi_{\mathtt{i}}^{-1})\circ_1
\eta_{\mathrm{F}}^{\Psi}\circ_1(\xi_{\mathtt{j}}\circ_0 \mathrm{id}_{\mathbf{M}(\mathrm{F})}).
\end{displaymath}
Then it is straightforward to check that this extends $\Phi$ to a $2$-natural transformation. 

\begin{proposition}\label{prop62}
Let $\mathbf{M}$ and $\mathbf{N}$ be two $2$-representations of $\cC$ and $\Psi:\mathbf{M}\to \mathbf{N}$ a 
$2$-natural transformation. Assume that for every $\mathtt{i}\in \cC$ the functor $\Psi_{\mathtt{i}}$ is an
equivalence.  Then there exists an inverse $2$-natural transformation.
\end{proposition}

\begin{proof}
For any $\mathtt{i}\in\cC$ choose an inverse equivalence  $\Phi_{\mathtt{i}}$ of $\Psi_{\mathtt{i}}$. Let
\begin{displaymath}
\xi_{\mathtt{i}}:\mathrm{Id}_{\mathbf{M}(\mathtt{i})}\to \Phi_{\mathtt{i}}\circ\Psi_{\mathtt{i}}\quad \text{ and }
\quad\zeta_{\mathtt{i}}: \Psi_{\mathtt{i}}\circ\Phi_{\mathtt{i}}\to\mathrm{Id}_{\mathbf{N}(\mathtt{i})} 
\end{displaymath}
be some isomorphisms. Define
\begin{displaymath}
\eta_{\mathrm{F}}^{\Phi}:=\big(
(\mathrm{id}_{\Phi_{\mathtt{j}}\circ\mathbf{N}(\mathrm{F})}\circ_0 \zeta_{\mathtt{i}})\circ_1
(\mathrm{id}_{\Phi_{\mathtt{j}}}\circ_0\eta_{\mathrm{F}}^{\Psi}\circ_0\mathrm{id}_{\Phi_{\mathtt{i}}})\circ_1
(\xi_{\mathtt{j}}\circ_0\mathrm{id}_{\mathbf{M}(\mathrm{F})\circ\Phi_{\mathtt{i}}})
\big)^{-1}.
\end{displaymath}
It is obvious that this produces a natural transformation, but we have to check that 
\begin{equation}\label{eq4}
 \eta_{\mathrm{F}\circ \mathrm{G}}^{\Phi}= (\mathrm{id}_{\mathbf{N}(\mathrm{F})}\circ_0\eta^{\Phi}_{\mathrm{G}})\circ_1
(\eta^{\Phi}_{\mathrm{F}}\circ_0\mathrm{id}_{\mathbf{M}(\mathrm{G})}).
\end{equation}
This follows from commutativity of the diagram
\begin{displaymath}
\xymatrix@C=-16mm{&&\mathbf{M}(\mathrm{F})\mathbf{M}(\mathrm{G})\Phi_{\mathtt{i}}\ar[dl] \ar[dr]&&&\\
& \Phi_{\mathtt{k}}\Psi_{\mathtt{k}}\mathbf{M}(\mathrm{F})\mathbf{M}(\mathrm{G})\Phi_{\mathtt{i}}\ar[dl] \ar[dr]  & &
\mathbf{M}(\mathrm{F})\Phi_{\mathtt{j}}\Psi_{\mathtt{j}}\mathbf{M}(\mathrm{G})\Phi_{\mathtt{i}} \ar[dl] \ar[dr]&&\\
\Phi_{\mathtt{k}}\mathbf{N}(\mathrm{F})\Psi_{\mathtt{j}}\mathbf{M}(\mathrm{G})\Phi_{\mathtt{i}} \ar[dr] &&
\Phi_{\mathtt{k}}\Psi_{\mathtt{k}}\mathbf{M}(\mathrm{F}) \Phi_{\mathtt{j}}\Psi_{\mathtt{j}}\mathbf{M}(\mathrm{G})\Phi_{\mathtt{i}} \ar[dl] \ar[dr]
&&
\mathbf{M}(\mathrm{F}) \Phi_{\mathtt{j}}\mathbf{N}(\mathrm{G})\Psi_{\mathtt{i}}\Phi_{\mathtt{i}}\ar[dl] \ar[dr]&\\&
\Phi_{\mathtt{k}}\mathbf{N}(\mathrm{F})\Psi_{\mathtt{j}}\Phi_{\mathtt{j}}\Psi_{\mathtt{j}}\mathbf{M}(\mathrm{G})\Phi_{\mathtt{i}}\ar[dl] \ar[dr]&&
\Phi_{\mathtt{k}}\Psi_{\mathtt{k}}\mathbf{M}(\mathrm{F}) \Phi_{\mathtt{j}}\mathbf{N}(\mathrm{G}) \Psi_{\mathtt{i}}\Phi_{\mathtt{i}}\ar[dl] \ar[dr]&&
\mathbf{M}(\mathrm{F}) \Phi_{\mathtt{j}}\mathbf{N}(\mathrm{G}) \ar[dl] \\
\Phi_{\mathtt{k}}\mathbf{N}(\mathrm{F})\Psi_{\mathtt{j}}\mathbf{M}(\mathrm{G})\Phi_{\mathtt{i}}\ar[dr]&& 
\Phi_{\mathtt{k}}\mathbf{N}(\mathrm{F}) \Psi_{\mathtt{j}}\Phi_{\mathtt{j}}\mathbf{N}(\mathrm{G}) \Psi_{\mathtt{i}}\Phi_{\mathtt{i}}\ar[dl] \ar[dr]&&
\Phi_{\mathtt{k}}\Psi_{\mathtt{k}}\mathbf{M}(\mathrm{F}) \Phi_{\mathtt{j}}\mathbf{N}(\mathrm{G})\ar[dl]&\\
& \Phi_{\mathtt{k}}\mathbf{N}(\mathrm{F})\mathbf{N}(\mathrm{G})\Psi_{\mathtt{i}}\Phi_{\mathtt{i}}  \ar[dr]&&\Phi_{\mathtt{k}}\mathbf{N}(\mathrm{F})\Psi_{\mathtt{j}}\Phi_{\mathtt{j}}\mathbf{N}(\mathrm{G}) \ar[dl]&&\\
&& \Phi_{\mathtt{k}}\mathbf{N}(\mathrm{F})\mathbf{N}(\mathrm{G})  &&&
}
\end{displaymath}
where the maps are the obvious ones (each of the maps has exactly one component of the form $\xi,\zeta$ or 
$\eta^\Psi$ and identities elsewhere). Commutativity of all squares is immediate. Then reading along the right border gives (the inverse of) the right hand side of \eqref{eq4}. Computing (the inverse of) the left hand side of \eqref{eq4} directly, using the definition of $\eta^\Phi$ and property \eqref{eq3} of  $\eta_{\mathrm{F}\circ \mathrm{G}}^{\Psi}$, gives the left border of the diagram, after noting that the third and fourth morphism in this path compose to the identity on $\Phi_{\mathtt{k}}\mathbf{N}(\mathrm{F})\Psi_{\mathtt{j}}\mathbf{M}(\mathrm{G})\Phi_{\mathtt{i}}$ by adjunction. Therefore \eqref{eq4} holds and this extends $\Phi$ to a $2$-natural transformation.
\end{proof}

In this scenario we will say that the $2$-representations $\mathbf{M}$ and $\mathbf{N}$ are {\em equivalent}.

\subsection{Abelianization and identities}\label{s1.35}

Denote by $\overline{\,\cdot\,}:\cC\text{-}\mathrm{afmod}\to \cC\text{-}\mathrm{mod}$ the abelianization
$2$-functor defined as in \cite[Subsection~4.2]{MM2}: for $\mathbf{M}\in \cC\text{-}\mathrm{afmod}$ and
$\mathtt{i}\in\cC$, the category $\overline{\mathbf{M}}(\mathtt{i})$ consists of all diagrams of the form
$X\overset{\alpha}{\longrightarrow}Y$, where $X,Y\in {\mathbf{M}}(\mathtt{i})$ and $\alpha$ is a morphism in
${\mathbf{M}}(\mathtt{i})$. Morphisms  in $\overline{\mathbf{M}}(\mathtt{i})$ are commutative squares modulo
factorization of the right downwards arrow using a homotopy. The $2$-action of $\cC$ on 
$\overline{\mathbf{M}}(\mathtt{i})$ is defined component-wise.

For any $2$-representation $\mathbf{M}$ of $\cC$ and any non-negative integer $k$, we denote by $\imorphism_k$ 
the $2$-natural transformation from $\mathbf{M}$ to $\mathbf{M}$ given by assigning to each $\mathtt{i}\in\cC$
the functor 
\begin{displaymath}
\underbrace{\mathrm{Id}_{\mathbf{M}(\mathtt{i})}\oplus \mathrm{Id}_{\mathbf{M}(\mathtt{i})}\oplus 
\dots \oplus \mathrm{Id}_{\mathbf{M}(\mathtt{i})}}_{k\,\text{summands}} 
\end{displaymath}
and defining $\eta_{\mathrm{F}}^{\imorphism_k}$ as $\mathrm{id}_{\mathrm{F}}\oplus \dots\oplus \mathrm{id}_{\mathrm{F}}$
(again with $k$ summands).

\subsection{Principal $2$-representations and additive subrepresentations}\label{s1.4}

For $\mathtt{i}\in\cC$ we denote by $\mathbf{P}_{\mathtt{i}}$ the principal $2$-representation
$\cC(\mathtt{i},{}_-)\in \cC\text{-}\mathrm{afmod}$. For any $\mathbf{M}\in\cC\text{-}\mathrm{amod}$
we have the usual Yoneda Lemma (see \cite[Subsection~2.1]{Le} and compare to \cite[Lemma~9]{MM2}):

\begin{lemma}\label{lem0}
\begin{equation}\label{eq1}
\mathrm{Hom}_{\ccC\text{-}\mathrm{amod}} (\mathbf{P}_{\mathtt{i}},\mathbf{M})\cong\mathbf{M}(\mathtt{i}).
\end{equation}
\end{lemma}

\begin{proof}
Let $\Psi:\mathbf{P}_{\mathtt{i}}\to \mathbf{M}$ be a $2$-natural transformation and set 
$X:=\Psi_{\mathtt{i}}(\mathbbm{1}_{\mathtt{i}})$. Denote by $\Phi:\mathbf{P}_{\mathtt{i}}\to \mathbf{M}$ the unique
strict $2$-natural transformation sending $\mathbbm{1}_{\mathtt{i}}$ to $X$ (see \cite[Lemma~9]{MM2}).
Then, for any $1$-morphism $\mathrm{F}\in\cC(\mathtt{i},\mathtt{j})$, we have the natural isomorphism
\begin{displaymath}
(\theta_{\mathtt{j}})_{\mathrm{F}}:=(\eta_{\mathrm{F}}^{\Psi})_{\mathbbm{1}_{\mathtt{i}}}:
\Psi_{\mathtt{j}}(\mathrm{F})\to \mathbf{M}(\mathrm{F})\,\Psi_{\mathtt{i}}(\mathbbm{1}_{\mathtt{i}})=
\mathbf{M}(\mathrm{F})\,X=\Phi_{\mathtt{j}}(\mathrm{F}).
\end{displaymath}
This gives us an (invertible) modification $\theta$ from $\Psi$ to $\Phi$ and the claim follows.
\end{proof}

Given $\mathbf{M}\in\cC\text{-}\mathrm{mod}$ and $X\in\mathbf{M}(\mathtt{i})$ for some $\mathtt{i}\in\cC$,
define $\mathbf{M}_{X}\in \cC\text{-}\mathrm{afmod}$ by restricting $\mathbf{M}$ to the full subcategories
$\mathrm{add}(\mathrm{F}\,X)$, where $\mathrm{F}$ runs through the set of all $1$-mor\-phisms in $\cC(\mathtt{i},\mathtt{j})$, $\mathtt{j}\in\cC$. 

\subsection{The multisemigroup of $\cC$ and cells}\label{s1.5}

The set $\mathcal{S}[\cC]$ of isomorphism classes of indecomposable $1$-morphisms in $\cC$ has the natural structure
of a multisemigroup induced by horizontal composition, see \cite[Subsection~3.1]{MM2} (see also \cite{KM} for more
details on multisemigroups). Let $\leq_L$, $\leq_R$ and $\leq_J$ denote the natural left, right and two-sided
orders on $\mathcal{S}[\cC]$, respectively. For example, $\mathrm{F}\leq_L \mathrm{G}$ means that for some 
$1$-morphism $\mathrm{H}$ the composition $\mathrm{H}\circ \mathrm{F}$ contains a direct summand isomorphic to 
$\mathrm{G}$. Equivalence classes with respect to $\leq_L$ are called {\em left cells}. Right and two-sided cells are 
defined analogously. Cells correspond exactly to Green's equivalence classes for the multisemigroup $\mathcal{S}[\cC]$.

A two-sided cell $\mathcal{J}$ is called {\em regular} if different left (right) cells in $\mathcal{J}$ are not
comparable with respect to the left (right) order. A two-sided cell $\mathcal{J}$ is called {\em strongly regular} 
if it is regular and, moreover, the intersection of any left and any right cell inside $\mathcal{J}$ consists
of exactly one element.

Given a left cell $\mathcal{L}$, there exists an $\mathtt{i}_{\mathcal{L}}\in\cC$ such that every $1$-morphism 
$\mathrm{F}\in \mathcal{L}$ belongs to $\cC(\mathtt{i}_{\mathcal{L}},\mathtt{j})$ for some $\mathtt{j}\in\cC$. 
Similarly, given a right cell $\mathcal{R}$, there exists a $\mathtt{j}_{\mathcal{R}}\in\cC$ such that every 
$1$-morphism $\mathrm{F}\in \mathcal{R}$ belongs to $\cC(\mathtt{i},\mathtt{j}_{\mathcal{R}})$ for some 
$\mathtt{i}\in\cC$.

\subsection{Cell $2$-representations}\label{s1.6}

Let $\mathcal{L}$ be a left cell and $\mathtt{i}=\mathtt{i}_{\mathcal{L}}$. Consider 
$\overline{\mathbf{P}}_{\mathtt{i}}$. For an indecomposable $1$-morphism $\mathrm{F}$ in some 
$\cC(\mathtt{i},\mathtt{j})$ denote by $L_{\mathrm{F}}$ the unique simple top of the indecomposable projective 
module $0\to \mathrm{F}$ in $\overline{\mathbf{P}}_{\mathtt{i}}(\mathtt{j})$. 
By \cite[Pro\-po\-si\-tion~17]{MM}, there exists a unique $\mathrm{G}_{\mathcal{L}}\in \mathcal{L}$ (called the
{\em Duflo involution} in $\mathcal{L}$) such that the indecomposable projective module $0\to\mathbbm{1}_{\mathtt{i}}$
has a unique quotient $N$ such that the simple socle of $N$ is isomorphic to $L_{\mathrm{G}_{\mathcal{L}}}$ and
$\mathrm{F}\, N/L_{\mathrm{G}_{\mathcal{L}}}=0$ for any $\mathrm{F}\in \mathcal{L}$. Set
$Q:=\mathrm{G}_{\mathcal{L}} \, L_{\mathrm{G}_{\mathcal{L}}}$. Then the additive $2$-representation
$\mathbf{C}_{\mathcal{L}}:=\left(\overline{\mathbf{P}}_{\mathtt{i}}\right)_{Q}$ is called the {\em additive cell} 
$2$-representation of $\cC$ associated to $\mathcal{L}$. The abelianization $\overline{\mathbf{C}}_{\mathcal{L}}$
of $\mathbf{C}_{\mathcal{L}}$ is called the {\em abelian cell} $2$-rep\-re\-sen\-ta\-ti\-on of $\cC$ associated 
to $\mathcal{L}$. For $\mathrm{F}\in \mathcal{L}$ we set $P_{\mathrm{F}}:=\mathrm{F}\, L_{\mathrm{G}_{\mathcal{L}}}$,
which we also identify with the indecomposable projective object $0\to \mathrm{F}\, L_{\mathrm{G}_{\mathcal{L}}}$
in $\overline{\mathbf{C}}_{\mathcal{L}}$.

\section{A special case of $2$-Schur's lemma}\label{s2}

In this section we prove a special case of Theorem~\ref{thm12} under one additional assumption of surjectivity
of the action of the center. It turns out that this assumption of surjectivity allows us to use
a short and elegant argument.

\subsection{The claim}\label{s2.1} 
The following is a special case of Theorem~\ref{thm12}:

\begin{theorem}\label{thm1}
Let $\cC$ be a fiat $2$-category, $\mathcal{J}$ a strongly regular two-sided cell of $\cC$ and $\mathcal{L}$ a
left cell in $\mathcal{J}$. Set $\mathtt{i}=\mathtt{i}_{\mathcal{L}}$ and $\mathrm{G}=\mathrm{G}_{\mathcal{L}}$. 
Assume that the natural map
\begin{equation}\label{eq2}
\begin{array}{ccc}
\mathrm{End}_{\ccC}(\mathbbm{1}_{\mathtt{i}})&\longrightarrow&
\mathrm{End}_{\mathbf{C}_{\mathcal{L}}}(P_{\mathrm{G}})\\
\varphi&\mapsto& \mathbf{C}_{\mathcal{L}}(\varphi)_{P_{\mathrm{G}}}
\end{array}
\end{equation}
is surjective. Then any endomorphism of $\mathbf{C}_{\mathcal{L}}$  is isomorphic to $\imorphism_k$ for some $k$ 
(in the category $\mathrm{End}_{\ccC\text{-}\mathrm{afmod}}(\mathbf{C}_{\mathcal{L}})$). Similarly, any endomorphism of $\overline{\mathbf{C}}_{\mathcal{L}}$  is isomorphic to $\imorphism_k$ for some $k$
(in the category $\mathrm{End}_{\ccC\text{-}\mathrm{mod}}(\overline{\mathbf{C}}_{\mathcal{L}})$).
\end{theorem}

\subsection{Annihilators of various objects in $\overline{\mathbf{C}}_{\mathcal{L}}$}\label{s2.2} 

For any $2$-representation $\mathbf{M}$ of $\cC$ and $X\in \mathbf{M}(\mathtt{j})$ for some $\mathtt{j}$, 
let $\mathrm{Ann}_{\ccC}(X)$ denote the left $2$-ideal of $\cC$ consisting of all $2$-mor\-phisms $\alpha$ which 
annihilate $X$. The key observation to prove Theorem~\ref{thm1} is the following:

\begin{lemma}\label{lem2}
Under the assumption of Theorem~\ref{thm1}, if $X\in \overline{\mathbf{C}}_{\mathcal{L}}(\mathtt{i})$ is such that 
$\mathrm{Ann}_{\ccC}(X)\supset \mathrm{Ann}_{\ccC}(L_{\mathrm{G}})$, then $X\in \mathrm{add}(L_{\mathrm{G}})$.
\end{lemma}

\begin{proof}
Let $\mathrm{F}\in\mathcal{L}$ be different from $\mathrm{G}$. Then $\mathrm{F}^*\, L_{\mathrm{F}}\neq 0$
by \cite[Lemma~15]{MM}. At the same time, from the fact that $\mathcal{J}$ is strongly simple it follows
that $\mathrm{F}^*\not\in \mathcal{L}$. Therefore $\mathrm{F}^*\, L_{\mathrm{G}}= 0$ by \cite[Lemma~15]{MM}.
Hence $\mathrm{id}_{\mathrm{F}^*}\in \mathrm{Ann}_{\ccC}(L_{\mathrm{G}})$ and at the same time
$\mathrm{id}_{\mathrm{F}^*}\not \in \mathrm{Ann}_{\ccC}(L_{\mathrm{F}})$.

Since $\mathrm{F}^*$ is exact, the previous paragraph implies that for any $X$ satisfying 
$\mathrm{Ann}_{\ccC}(X)\supset \mathrm{Ann}_{\ccC}(L_{\mathrm{G}})$,  every simple subquotient of 
$X$ is isomorphic to $L_{\mathrm{G}}$. Assume now that $X$ is indecomposable such that there is a short exact
sequence
\begin{displaymath}
0\to  L_{\mathrm{G}}\to X\to L_{\mathrm{G}}\to 0.
\end{displaymath}
Then there is a short exact sequence $K\hookrightarrow P_{\mathrm{G}}\tto X$ and an endomorphism of
$P_{\mathrm{G}}$ which induces a non-trivial nilpotent endomorphism of $X$. From \eqref{eq2}, it follows that the
natural map
\begin{displaymath}
\begin{array}{ccc}
\mathrm{End}_{\ccC}(\mathbbm{1}_{\mathtt{i}})&\longrightarrow&
\mathrm{End}_{\mathbf{C}_{\mathcal{L}}}(X)\\
\varphi&\mapsto& \mathbf{C}_{\mathcal{L}}(\varphi)_{X}
\end{array}
\end{displaymath}
is surjective. Let $\alpha\in \mathrm{End}_{\ccC}(\mathbbm{1}_{\mathtt{i}})$ be a $2$-morphism which produces a
non-trivial nilpotent endomorphism of $X$. Then $\alpha\not \in \mathrm{Ann}_{\ccC}(X)$ while
$\alpha^2 \in \mathrm{Ann}_{\ccC}(X)$. At the same time, $\mathrm{End}_{\ccC}(\mathbbm{1}_{\mathtt{i}})$ is a 
local finite dimensional $\Bbbk$-algebra (see Subsection~\ref{s1.2}), and hence $\alpha$ is either nilpotent 
or invertible. But $\alpha$ cannot be invertible as $\alpha^2$ annihilates $X$. Therefore, $\alpha$ is nilpotent.
This implies that $\alpha\in \mathrm{Ann}_{\ccC}(L_{\mathrm{G}})$ as any nonzero endomorphism of 
$L_{\mathrm{G}}$ is invertible by Schur's lemma. 

Finally, if $Y$ is an indecomposable module, every simple subquotient of which is isomorphic to $L_{\mathrm{G}}$,
then $Y$ has a subquotient $X$ as in the previous paragraph. Therefore 
$\mathrm{Ann}_{\ccC}(L_{\mathrm{G}})\not\subset \mathrm{Ann}_{\ccC}(Y)$. The claim of the lemma follows. 
\end{proof}

\subsection{Proof of Theorem~\ref{thm1}}\label{s2.3} 

Let $\Psi\in \mathrm{End}_{\ccC\text{-}\mathrm{mod}}(\overline{\mathbf{C}}_{\mathcal{L}})$. By Lemma~\ref{lem2},
we have $\Psi_{\mathtt{i}}(L_{\mathrm{G}})\cong L_{\mathrm{G}}^{\oplus k}$ for some non-negative integer $k$.
Now for any $\mathrm{F}\in\mathcal{L}$ we have an isomorphism 
\begin{displaymath}
\Psi_{\mathtt{j}}(P_{\mathrm{F}})=\Psi_{\mathtt{j}}(\mathrm{F}\, L_{\mathrm{G}}) \cong
\mathrm{F}\,L_{\mathrm{G}}^{\oplus k}\cong P_{\mathrm{F}}^{\oplus k},
\end{displaymath}
natural in $\mathrm{F}$. As $\Psi_{\mathtt{j}}$ is right exact, every indecomposable projective is of the form
$P_{\mathrm{F}}$, and $2$-morphisms in $\cC$ surject onto homomorphisms between indecomposable projectives
(see \cite[Subsection~4.5]{MM}), we have that $\Psi_{\mathtt{j}}$ is isomorphic to $\mathrm{Id}_{\overline{\mathbf{C}}_{\mathcal{L}}(\mathtt{j})}^{\oplus k}$. Clearly, $k$ does not depend
on $\mathtt{j}$. Now we repeat the argument from the proof of Lemma~\ref{lem0}. We have the natural isomorphisms
\begin{displaymath}
(\theta_{\mathtt{j}})_{\mathrm{F}\,L_{\mathrm{G}}}:=
(\eta_{\mathrm{F}}^{\Psi})_{L_{\mathrm{G}}}: 
\Psi_{\mathtt{j}}\circ\overline{\mathbf{C}}_{\mathcal{L}}(\mathrm{F})\,L_{\mathrm{G}}\to
\overline{\mathbf{C}}_{\mathcal{L}}(\mathrm{F})\circ (\imorphism_k)_{\mathtt{i}}\, L_{\mathrm{G}}=
\overline{\mathbf{C}}_{\mathcal{L}}(\mathrm{F})\, L_{\mathrm{G}}^{\oplus k},
\end{displaymath}
which give us an invertible modification $\theta$ from $\Psi$ to $\imorphism_k$. This proves the
abelian part of Theorem~\ref{thm1}.

To prove the additive part we just note that any 
$\Psi\in \mathrm{End}_{\ccC\text{-}\mathrm{mod}}({\mathbf{C}}_{\mathcal{L}})$ abelianizes to 
$\overline{\Psi}\in \mathrm{End}_{\ccC\text{-}\mathrm{mod}}(\overline{\mathbf{C}}_{\mathcal{L}})$.
Now the additive claim of Theorem~\ref{thm1} follows from the abelian claim by restricting to 
projective modules. \hfill$\square$

\section{Description of $\mathcal{J}$-simple fiat $2$-categories}\label{s3}

\subsection{Definition of $2$-full $2$-representations}\label{s3.1}

Let $\cC$ be a finitary category and $\mathbf{M}$ a $2$-rep\-re\-sen\-ta\-ti\-on of $\cC$. We will say that 
$\mathbf{M}$ is {\em $2$-full} provided that for any $1$-morphisms $\mathrm{F},\mathrm{G}\in\cC$ the 
representation map
\begin{equation}\label{eq8}
\mathrm{Hom}_{\ccC}(\mathrm{F},\mathrm{G})\to 
\mathrm{Hom}_{\mathfrak{X}}(\mathbf{M}(\mathrm{F}),\mathbf{M}(\mathrm{G})),
\end{equation}
where $\mathfrak{X}\in\{\mathfrak{A}_{\Bbbk},\mathfrak{A}_{\Bbbk}^f,\mathfrak{R}_{\Bbbk}\}$ is the target
$2$-category of $\mathbf{M}$, is surjective. In other words, $2$-morphisms in $\cC$ surject onto the 
space of natural transformations between functors.

If $\mathcal{J}$ is a $2$-sided cell of $\cC$, we will say that $\mathbf{M}$ is {\em $\mathcal{J}$-$2$-full} 
provided that for any $1$-morphisms $\mathrm{F},\mathrm{G}\in\mathcal{J}$ the  representation map
\eqref{eq8} is surjective.

\subsection{The $2$-category associated with $\mathcal{J}$}\label{s3.2}

Let now $\cC$ be a fiat $2$-category and $\mathcal{J}$ a two-sided cell in $\cC$.  Let $\mathcal{L}$ be a left cell 
of $\mathcal{J}$, $\mathrm{G}:=\mathrm{G}_{\mathcal{L}}$ and $\mathtt{i}:=\mathtt{i}_{\mathcal{L}}$. Let $\cJ$ be 
the unique maximal $2$-ideal of $\cC$ which does not contain $\mathrm{id}_{\mathrm{F}}$ for any 
$\mathrm{F}\in \mathcal{J}$ (see \cite[Theorem~15]{MM2}). Then the quotient $2$-category $\cC/\cJ$ is 
$\mathcal{J}$-simple (see \cite[Subsection~6.2]{MM2}). Denote by $\cC^{(\mathcal{J})}$ the $2$-full $2$-subcategory of 
$\cC/\cJ$ generated by $\mathbbm{1}_{\mathtt{i}_{\mathcal{L}}}$ and all $\mathrm{F}\in \mathcal{J}$ (and closed 
with respect to isomorphism of $1$-morphisms). We will call $\cC^{(\mathcal{J})}$ the {\em $\mathcal{J}$-simple 
$2$-category associated to $\mathcal{J}$}.

The cell $2$-representation $\mathbf{C}_{\mathcal{L}}$ of $\cC$ factors over $\cC/\cJ$ by \cite[Theorem~19]{MM2}
and hence restricts to a $2$-representation of $\cC^{(\mathcal{J})}$. Assume now that $\mathcal{J}$ is strongly 
regular. Then, by \cite[Proposition~32]{MM}, $\mathcal{J}$ remains a strongly regular two-sided cell 
in $\cC^{(\mathcal{J})}$. Moreover, using \cite[Subsection~6.5]{MM2}, the restriction of $\mathbf{C}_{\mathcal{L}}$ 
to  $\cC^{(\mathcal{J})}$ is equivalent to the corresponding cell $2$-representation of $\cC^{(\mathcal{J})}$.

\begin{center}
{\em For the remainder of this section we fix a strongly regular cell $\mathcal{J}$\\
and assume that $\cC=\cC^{(\mathcal{J})}$.}
\end{center}

\subsection{Detecting $2$-fullness}\label{s3.3}

We consider the cell $2$-representation $\mathbf{M}:=\overline{\mathbf{C}}_{\mathcal{L}}$.
We start our analysis with the following observation:

\begin{proposition}\label{prop4}
For $\mathrm{F}\in \mathcal{J}$ and $\mathtt{j}\in\cC$ consider the representation map 
\begin{equation}\label{eq5}
\mathrm{Hom}_{\ccC}(\mathrm{F},\mathbbm{1}_{\mathtt{j}})\to  
\mathrm{Hom}_{\mathfrak{R}_{\Bbbk}}(\mathbf{M}(\mathrm{F}),\mathbf{M}(\mathbbm{1}_{\mathtt{j}})).
\end{equation}
If this map is surjective for $\mathrm{F}=\mathrm{G}$ and $\mathtt{j}=\mathtt{i}$, then it is 
surjective for any $\mathrm{F}$ and $\mathtt{j}$.
\end{proposition}

Note that both sides of \eqref{eq5} are empty unless $\mathrm{F}\in\cC(\mathtt{j},\mathtt{j})$.
As usual, to simplify notation we will use the module notation and write $\mathrm{F}\, X$ instead of
$\mathbf{M}(\mathrm{F})(X)$.

\begin{proof}
Let $\mathrm{H},\mathrm{K}\in \mathcal{L}$ and assume that $\mathrm{H},\mathrm{K}\in \cC(\mathtt{i},\mathtt{j})$. 
By strong regularity of $\mathcal{J}$ we have 
$\mathrm{H}\mathrm{K}^*=a\mathrm{F}$ for some $\mathrm{F}\in \mathcal{J}$ and $a\in\mathbb{N}$,
moreover, if we vary $\mathrm{H}$ and $\mathrm{K}$, we can obtain any $\mathrm{F}\in \mathcal{J}$ in this way.
To see that $\mathrm{H}\mathrm{K}^*\neq 0$, one evaluates $\mathrm{H}\mathrm{K}^*$ on $L_{\mathrm{K}}$
obtaining $\mathrm{K}^*L_{\mathrm{K}}=P_{\mathrm{G}}$ (by \cite[Corollary~38(a)]{MM}), and 
$\mathrm{H}P_{\mathrm{G}}\neq 0$ since $\mathrm{H}L_{\mathrm{G}}=P_{\mathrm{H}}\neq 0$.

Similarly, we have $\mathrm{K}^*\mathrm{H}=b\mathrm{G}$ for some $b\in\mathbb{N}$ since $\mathrm{K}^*\mathrm{H}$ 
is in the same left cell as $\mathrm{H}$ (which is $\mathcal{L}$) and the same right cell as $\mathrm{K}^*$ 
(which is $\mathcal{L}^*$), and $\mathcal{L}\cap\mathcal{L}^*=\{\mathrm{G}\}$ since $\mathcal{J}$ is strongly regular. 
Using the involution $*$ we have
\begin{displaymath}
\mathrm{Hom}_{\ccC}(\mathrm{H},\mathrm{K})\cong \mathrm{Hom}_{\ccC}(\mathrm{K}^*,\mathrm{H}^*).
\end{displaymath}
By adjunction, we have
\begin{equation}\label{eqn3}
\mathrm{Hom}_{\ccC}(\mathrm{H},\mathrm{K})\cong b\mathrm{Hom}_{\ccC}(\mathrm{G},\mathbbm{1}_{\mathtt{i}}),\quad\quad
\mathrm{Hom}_{\ccC}(\mathrm{K}^*,\mathrm{H}^*)\cong a\mathrm{Hom}_{\ccC}(\mathrm{F},\mathbbm{1}_{\mathtt{j}}).
\end{equation}

Evaluating $\mathrm{Hom}_{\ccC}(\mathrm{H},\mathrm{K})$ at $L_{\mathrm{G}}$ 
(which is surjective by \cite[Subsection~4.5]{MM}) and using adjunction, we get
\begin{displaymath}
\mathrm{Hom}_{\mathbf{M}(\mathtt{j})}(\mathrm{H}\,L_{\mathrm{G}},\mathrm{K}\,L_{\mathrm{G}}) \cong
b\mathrm{Hom}_{\mathbf{M}(\mathtt{i})}(\mathrm{G}\,L_{\mathrm{G}},L_{\mathrm{G}}).
\end{displaymath}
As $\mathrm{G}\,L_{\mathrm{G}}\cong P_{\mathrm{G}}$, the space 
$\mathrm{Hom}_{\mathbf{M}(\mathtt{i})}(\mathrm{G}\,L_{\mathrm{G}},L_{\mathrm{G}})$ is one-dimensional, and 
thus 
\begin{equation}\label{eqn1}
b=\dim \mathrm{Hom}_{\mathbf{M}(\mathtt{j})}(\mathrm{H}\,L_{\mathrm{G}},\mathrm{K}\,L_{\mathrm{G}})
\end{equation}
 
On the other hand, evaluating $\mathrm{Hom}_{\ccC}(\mathrm{K}^*,\mathrm{H}^*)$ at a multiplicity free direct sum 
$L$ of all simple modules in $\mathbf{M}(\mathtt{j})$ and using adjunction, we have
\begin{equation}\label{eqn2}
\mathrm{Hom}_{\mathbf{M}(\mathtt{i})}(\mathrm{K}^*\,L,\mathrm{H}^*\,L) \cong
a\mathrm{Hom}_{\mathbf{M}(\mathtt{j})}(\mathrm{F}\,L,L).
\end{equation}
By \cite[Lemma~12]{MM}, $\mathrm{K}^*\,L_{\mathrm{Q}} \neq 0$ for a direct summand $L_{\mathrm{Q}}$ of $L$, labeled by $\mathrm{Q} \in \mathcal{L}$, implies that $\mathrm{K}$ is in the same right cell as $\mathrm{Q}$. Strong regularity implies $\mathrm{Q}=\mathrm{K}$ and by \cite[Corollary~38(a)]{MM}, we have $\mathrm{K}^*\,L \cong P_\mathrm{G}$. Similarly $\mathrm{H}^*\,L \cong P_\mathrm{G}$ and the left hand side of \eqref{eqn2} is isomorphic to $\mathrm{End}_{\mathbf{M}(\mathtt{i})}(P_{\mathrm{G}})$.

As $\mathrm{F}$ is a direct summand of $\mathrm{H}\mathrm{K}^*$, again $L_{\mathrm{K}}$ is the only simple module which 
is not annihilated by $\mathrm{F}$. By \cite[Theorem~31]{MM5}, the module $\mathrm{F}\,L_{\mathrm{K}}$ is an 
indecomposable projective in $\mathbf{M}(\mathtt{j})$, namely $P_{\mathrm{H}}$. This means that 
$\dim \mathrm{Hom}_{\mathbf{M}(\mathtt{j})}(\mathrm{F}\,L,L)=1$ and hence 
\begin{equation}\label{eqn4}
a=\dim \mathrm{End}_{\mathbf{M}(\mathtt{i})}(P_{\mathrm{G}}).
\end{equation}
To proceed we need the following claim:

\begin{lemma}\label{lem5}
Let $A$ be a finite dimensional $\Bbbk$-algebra and $e, f\in A$ primitive idempotents. Assume that $\mathrm{F}$ is an 
exact endofunctor of $A\text{-}\mathrm{mod}$ such that $\mathrm{F}\,L_f\cong Ae$ and $\mathrm{F}\,L_g=0$ for any 
simple $L_g\not\cong L_f$. Then $\mathrm{F}$ is isomorphic to the functor $\mathrm{F}'$ given by tensoring with 
the bimodule $Ae\otimes_{\Bbbk} fA$ and, moreover, 
\begin{displaymath}
\mathrm{Hom}_{\mathfrak{R}_{\Bbbk}}(\mathrm{F},\mathrm{Id}_{A\text{-}\mathrm{mod}})\cong \mathrm{Hom}_A(Ae, Af). 
\end{displaymath}
\end{lemma}

\begin{proof}
Let $L$ be a multiplicity free sum of all simple $A$-modules. As $\mathrm{F}\,L_f$ has simple top $L_e$, it follows
that $\mathrm{F}$ is a quotient of $\mathrm{F}'$, which gives us a surjective natural transformation
$\alpha:\mathrm{F}'\to \mathrm{F}$. Further, $\mathrm{F}\,L\cong \mathrm{F}'\,L$, meaning that 
$\alpha$ is an isomorphism when evaluated on simple modules. Using induction on the length of a module and the 
3-Lemma we obtain that $\alpha$ is an isomorphism, which proves the first claim. The second claim follows by adjunction.
\end{proof}

From Lemma~\ref{lem5} and surjectivity of \eqref{eq5} for $\mathrm{G}$,  we get
\begin{displaymath}
\dim \mathrm{Hom}_{\ccC}(\mathrm{G},\mathbbm{1}_{\mathtt{i}})
=\dim \mathrm{End}_{\mathbf{M}(\mathtt{i})}(P_{\mathrm{G}}).
\end{displaymath}
Using \eqref{eqn3}, \eqref{eqn1} and Lemma~\ref{lem5}, we have
\begin{equation*}
\begin{split}
\dim \mathrm{Hom}_{\ccC}(\mathrm{H},\mathrm{K})&=
\dim \mathrm{Hom}_{\mathbf{M}(\mathtt{j})}(\mathrm{H}\,L_{\mathrm{G}},\mathrm{K}\,L_{\mathrm{G}})\cdot 
\dim \mathrm{End}_{\mathbf{M}(\mathtt{i})}(P_{\mathrm{G}})\\&=
\dim \mathrm{Hom}_{\mathbf{M}(\mathtt{j})}(P_\mathrm{H},P_\mathrm{K})\cdot 
\dim \mathrm{End}_{\mathbf{M}(\mathtt{i})}(P_{\mathrm{G}}).
\end{split}
\end{equation*}
On the other hand, using \eqref{eqn3} and \eqref{eqn4} we have 
 \begin{displaymath}
\dim \mathrm{Hom}_{\ccC}(\mathrm{K}^*,\mathrm{H}^*) =
\dim \mathrm{Hom}_{\ccC}(\mathrm{F},\mathbbm{1}_{\mathtt{j}})\cdot 
\dim \mathrm{End}_{\mathbf{M}(\mathtt{i})}(P_{\mathrm{G}}).
\end{displaymath}
As $\cC$ is $\mathcal{J}$-simple, $\dim \mathrm{Hom}_{\ccC}(\mathrm{F},\mathbbm{1}_{\mathtt{j}}) \leq \dim \mathrm{Hom}_{\mathfrak{R}_{\Bbbk}}(\mathbf{M}(\mathrm{F}),\mathbf{M}(\mathbbm{1}_{\mathtt{j}})) $ and the latter by Lemma~\ref{lem5} is equal to $\dim \mathrm{Hom}_{\mathbf{M}(\mathtt{j})}(P_\mathrm{H},P_\mathrm{K})$.
Dividing through by $\dim \mathrm{End}_{\mathbf{M}(\mathtt{i})}(P_{\mathrm{G}})$ yields 
\begin{equation*}
\begin{split}
\dim \mathrm{Hom}_{\mathbf{M}(\mathtt{j})}(P_\mathrm{H},P_\mathrm{K})&= \dim \mathrm{Hom}_{\ccC}(\mathrm{F},\mathbbm{1}_{\mathtt{j}}) \\ &\leq  \dim \mathrm{Hom}_{\mathfrak{R}_{\Bbbk}}(\mathbf{M}(\mathrm{F}),\mathbf{M}(\mathbbm{1}_{\mathtt{j}})) \\
&= \dim \mathrm{Hom}_{\mathbf{M}(\mathtt{j})}(P_\mathrm{H},P_\mathrm{K})\end{split}
\end{equation*}
and hence
\begin{displaymath}
\dim \mathrm{Hom}_{\ccC}(\mathrm{F},\mathbbm{1}_{\mathtt{j}}) =  
\dim \mathrm{Hom}_{\mathfrak{R}_{\Bbbk}}(\mathbf{M}(\mathrm{F}),\mathbf{M}(\mathbbm{1}_{\mathtt{j}})).
\end{displaymath}
Injectivity of the representation map, which follows from $\mathcal{J}$-simplicity of $\cC$, now implies surjectivity and hence the statement of the proposition.
\end{proof}

\begin{proposition}\label{prop6}
Let $\mathrm{H},\mathrm{K}\in\cC(\mathtt{j},\mathtt{k})\cap\mathcal{J}$. If the representation map
\eqref{eq5} is surjective for $\mathrm{F}=\mathrm{G}$ and $\mathtt{i}=\mathtt{j}$, then the representation map
\begin{displaymath}
\mathrm{Hom}_{\ccC}(\mathrm{H},\mathrm{K})\to \mathrm{Hom}_{\mathfrak{R}_{\Bbbk}}(\mathbf{M}(\mathrm{H}),
\mathbf{M}(\mathrm{K}))
\end{displaymath}
is surjective. 
\end{proposition}

\begin{proof}
As $\mathcal{J}$ is strongly regular, we have $\mathrm{K}^*\mathrm{H}=\mathrm{Q}^{\oplus m}$ for some
$m\in\mathbb{N}_0$, where $\mathrm{Q}$ is in the intersection of the left cell of $\mathrm{H}$
and the right cell of $\mathrm{K}^*$. We have the commutative diagram
\begin{displaymath}
\xymatrix@C=2mm{ 
\mathrm{Hom}_{\ccC}(\mathrm{H},\mathrm{K})\ar[d]\ar[rr]^{\sim} &&
\mathrm{Hom}_{\ccC}(\mathrm{K}^*\mathrm{H},\mathbbm{1}_{\mathtt{j}})\ar[d]\ar[rr]^{\sim}
&&\mathrm{Hom}_{\ccC}(\mathrm{Q},\mathbbm{1}_{\mathtt{j}})^{\oplus m}\ar[d]\\
\mathrm{Hom}_{\mathfrak{R}_{\Bbbk}}(\mathbf{M}(\mathrm{H}),\mathbf{M}(\mathrm{K}))\ar[rr]^{\sim} &&
\mathrm{Hom}_{\mathfrak{R}_{\Bbbk}}(\mathbf{M}(\mathrm{K}^*\mathrm{H}),\mathrm{Id}_{\mathbf{M}(\mathtt{j})})
\ar[rr]^{\sim}
&&\mathrm{Hom}_{\mathfrak{R}_{\Bbbk}}(\mathbf{M}(\mathrm{Q}),
\mathrm{Id}_{\mathbf{M}(\mathtt{j})})^{\oplus m}
}
\end{displaymath}
where the vertical arrows are the representation maps, the left horizontal arrows are isomorphisms given by
adjunction, and the right horizontal arrows are isomorphisms given by additivity. Then the rightmost vertical arrow
is an isomorphism by Proposition~\ref{prop4} and $\mathcal{J}$-simplicity of $\cC$. This implies that all vertical
arrows are isomorphisms and the claim follows.
\end{proof}

\subsection{Cell $2$-representations are $\mathcal{J}$-$2$-full}\label{s3.4}

${}$

{\color{red} The following theorem is wrong in general, and holds only under the following assumption.

\begin{assum}\label{bigassum} 
Let $\alpha\colon \mathrm{G} \to \mathbbm{1}_{\mathtt{i}}$ be the 
morphism defining the Duflo involution  
(cf. \cite[Proposition~17]{MM})
and $\bar\alpha$ its mate under the adjunction isomorphism
\[ \mathrm{Hom}_{\ccC}(\mathrm{G},\mathbbm{1}_{\mathtt{i}})\cong  \mathrm{Hom}_{\ccC}(\mathbbm{1}_{\mathtt{i}},\mathrm{G})\] 
(where we use $\mathrm{G}^*\cong \mathrm{G}$).
Assume that the composition $\mathbbm{1}_{\mathtt{i}} \xrightarrow{\bar\alpha} \mathrm{G} \xrightarrow{\alpha} \mathbbm{1}_{\mathtt{i}} $ is nonzero. 
\end{assum}

This assumption always holds in characteristic $0$. A counterexample to the general statement and a proof of Theorem \ref{thm7} under the additional assumption is given in Section \ref{corrig}.}

\begin{theorem}\label{thm7}
The cell $2$-representation $\mathbf{M}:=\overline{\mathbf{C}}_{\mathcal{L}}$ is $\mathcal{J}$-$2$-full.
\end{theorem}

\begin{proof}
Thanks to Proposition~\ref{prop6}, we have only to show that the representation map \eqref{eq5} is surjective 
for $\mathrm{F}=\mathrm{G}$ and $\mathtt{i}=\mathtt{j}$. In order to show this it suffices, by Lemma~\ref{lem5} 
and $\mathcal{J}$-simplicity of $\cC$, to show that 
\begin{displaymath}
\dim \mathrm{Hom}_{\ccC}(\mathrm{G},\mathbbm{1}_{\mathtt{i}})
=\dim \mathrm{End}_{\mathbf{M}(\mathtt{i})}(P_{\mathrm{G}}).
\end{displaymath}
By Lemma~\ref{lem5} and $\mathcal{J}$-simplicity of $\cC$, we have 
\begin{displaymath}
\dim \mathrm{Hom}_{\ccC}(\mathrm{G},\mathbbm{1}_{\mathtt{i}})
\leq \dim \mathrm{End}_{\mathbf{M}(\mathtt{i})}(P_{\mathrm{G}}).
\end{displaymath}

Recall from \cite[Proposition~17]{MM} that there is a unique submodule $K$ of the indecomposable projective module
$0\to\mathbbm{1}_{\mathtt{i}}$ in $\overline{\mathbf{P}}_{\mathtt{i}}(\mathtt{i})$ which has simple top
$L_{\mathrm{G}}$ and such that the quotient of the projective by $K$ is annihilated by $\mathrm{G}$.
We denote by $\beta$ some $2$-morphism from $\mathrm{G}$ to $\mathbbm{1}_{\mathtt{i}}$ which gives rise to a
surjection from $0\to \mathrm{G}$ to $K$ in $\overline{\mathbf{P}}_{\mathtt{i}}(\mathtt{i})$. Then the
$\mathrm{End}_{\ccC}(\mathrm{G})$-module $\mathrm{Hom}_{\ccC}(\mathrm{G},\mathbbm{1}_{\mathtt{i}})$ has simple
top and $\beta$ is a representative for this simple top. 

Let $A$ be a basic finite dimensional associative $\Bbbk$-algebra such that 
$\mathbf{M}(\mathtt{i})\cong A\text{-}\mathrm{mod}$. Let $1=\sum_{i=1}^n e_i$ be a decomposition of $1\in A$ into
a sum of pairwise orthogonal primitive idempotents. We assume that $e=e_1$ is a primitive idempotent corresponding to
$L_{\mathrm{G}}$. From Lemma~\ref{lem5}, we have that the functor $\mathbf{M}(\mathrm{G})$ is isomorphic to tensoring 
with $Ae\otimes_{\Bbbk} eA$. Clearly, $\mathbf{M}(\mathbbm{1}_{\mathtt{i}})$ is isomorphic to tensoring 
with $A$. 

Since $\mathcal{J}$ is strongly regular, Duflo involutions in $\mathcal{J}\cap\cC(\mathtt{i},\mathtt{i})$ are in 
bijection with $\{e_1,e_2,\dots,e_n\}$. Let $\mathrm{G}_i$ be the Duflo involution corresponding to $e_i$. Similarly
to the existence of $\beta$, there is a $\beta_i$ for each $i$, which we can put into the $2$-morphism
\begin{displaymath}
\gamma:=(\beta_1,\beta_2,\dots,\beta_n):\bigoplus_i \mathrm{G}_i\to \mathbbm{1}_{\mathtt{i}}. 
\end{displaymath}
The cokernel $\mathrm{Coker}(\gamma)$, as an object of $\overline{\mathbf{P}}_{\mathtt{i}}$, is annihilated by all 
$1$-morphisms in $\mathcal{J}$. This implies that $\mathbf{M}(\mathrm{Coker}(\gamma))$ annihilates $L_{\mathrm{F}}$
for every $\mathrm{F}\in \mathcal{L}$ and hence $\mathbf{M}(\mathrm{Coker}(\gamma))=0$ by right exactness of
$\mathbf{M}(\mathrm{Coker}(\gamma))$. From this we derive that $\mathbf{M}(\gamma)$ is surjective and hence
we can choose $\beta$ and the above identifications of functors with bimodules such that $\mathbf{M}(\beta)$ is 
the multiplication map  $Ae\otimes_{\Bbbk}eA\to A$.

In order to show that $\dim \mathrm{Hom}_{\ccC}(\mathrm{G},\mathbbm{1}_{\mathtt{i}})
\geq \dim \mathrm{End}_{\mathbf{M}(\mathtt{i})}(P_{\mathrm{G}})$, we show that no 
$\varphi \in \mathrm{End}_{\ccC}(\mathrm{G})$ that induces a nonzero endomorphism of $P_{\mathrm{G}}$ when 
evaluated at $L_{\mathrm{G}}$, is sent to zero under composition with $\beta$.

In order to see this, let $\varphi \in \mathrm{End}_{\ccC}(\mathrm{G})$ be such that 
$\mathbf{M}(\varphi) \in eAe\otimes eAe$ is not killed under the map 
$eAe\otimes eAe \twoheadrightarrow eAe\otimes eAe/\mathrm{Rad}(eAe) \cong eAe$. In other words, writing 
$\mathbf{M}(\varphi)= \sum_j (\psi_j \otimes (c_je + r_j))$ for some $c_j \in\Bbbk, r_j \in \mathrm{Rad} (eAe)$, 
and where $\psi_j$ runs over a basis of $eAe$, chosen in accordance with radical powers, we have that 
$\psi:=\sum_j c_j\psi_j$ is nonzero in $eAe$. Then $\mathbf{M}(\beta\circ\varphi) = \psi + (\sum_j c_j\psi_j r_j) 
\in eAe$. As $\psi \in \mathrm{Rad}^k (eAe)$ implies $\psi_j \in \mathrm{Rad}^k (eAe)$ for all $\psi_j$ such 
that  $c_j \neq 0$, the summand $\sum_j c_j\psi_j r_j$ is in $\mathrm{Rad}^{k+1} (eAe)$ and hence $\mathbf{M}(\beta\circ\varphi)\in \mathrm{Hom}_{\mathfrak{R}_\Bbbk}(\mathbf{M}(\mathrm{G}),\mathbf{M}(\mathbbm{1}_{\mathtt{i}}))$
is nonzero.
Therefore $\beta\circ\varphi\in \mathrm{Hom}_{\ccC}(\mathrm{G},\mathbbm{1}_{\mathtt{i}})$ is nonzero for any $\varphi \in \mathrm{End}_{\ccC}(\mathrm{G})$ that is not killed by evaluation at $L_{\mathrm{G}}$. By surjectivity of the map  from  $\mathrm{End}_{\ccC}(\mathrm{G})$  onto $\mathrm{End}_{\mathbf{M}(\mathtt{i})}(P_{\mathrm{G}})$ given by evaluation at $L_{\mathrm{G}}$ (see \cite[Subsection~4.5]{MM}), this implies 
\begin{displaymath}
\dim \mathrm{Hom}_{\ccC}(\mathrm{G},\mathbbm{1}_{\mathtt{i}})
\geq \dim \mathrm{End}_{\mathbf{M}(\mathtt{i})}(P_{\mathrm{G}}) 
\end{displaymath}
and completes the proof of the proposition.
\end{proof}

\begin{corollary}\label{cor8}
Assume that $\cC$ is any fiat $2$-category and $\mathcal{J}$ is a strongly regular $2$-sided cell of $\cC$. {\color{red} Assume that Assumption \ref{bigassum} is satisfied.}
Then for any left cell $\mathcal{L}$ in $\mathcal{J}$ the cell $2$-representation 
$\overline{\mathbf{C}}_{\mathcal{L}}$ is $\mathcal{J}$-$2$-full.
\end{corollary}

\begin{proof}
This follows directly from Theorem~\ref{thm7} and \cite[Corollary~33]{MM}. 
\end{proof}

\subsection{Construction of $\mathcal{J}$-simple $2$-categories $\cC^{(\mathcal{J})}$}\label{s3.5}

Let $n\in\mathbb{N}$ and $A:=(A_1,A_2,\dots,A_n)$ be a collection of pairwise non-isomorphic, basic, connected, 
weakly symmetric finite dimensional associative $\Bbbk$-algebras. For $i\in\{1,2,\dots,n\}$ choose some small 
category $\mathcal{C}_i$ equivalent to $A_i\text{-}\mathrm{mod}$, and let $Z_i$ denote the center of $A_i$.
Set $\mathcal{C}=(\mathcal{C}_1,\mathcal{C}_2,\dots,\mathcal{C}_n)$.
Denote by $\cC_{\mathcal{C}}$ the $2$-full fiat $2$-subcategory of $\mathfrak{R}_{\Bbbk}$ with objects
$\mathcal{C}_i$, which is closed under isomorphisms of $1$-morphisms and generated by functors that 
are isomorphic to tensoring with projective $A_i\text{-}A_j$ bimodules.

We identify $Z_i$ with $\mathrm{End}_{\ccC_{\hspace{-2pt}\mathcal{C}}}(\mathbbm{1}_{\mathcal{C}_i})$ and
denote by $Z'_i$ the subalgebra of $Z_i$ generated by $\mathrm{id}_{\mathbbm{1}_{\mathcal{C}_i}}$ and all 
elements which factor through $1$-morphisms given by tensoring with projective $A_i\text{-}A_i$ bimodules.

\begin{remark}\label{rem9}
{ 
In general, $Z'_i\neq Z_i$. For example, let $n=1$ and $A=A_1=\Bbbk[x]/(x^3)$. Then $Z=Z_1=A$ while
$Z'_1$ is the linear span of $1$ and $x^2$ in $Z$. Indeed, we have only one projective bimodule
$A\otimes_{\Bbbk}A$, which  has Loewy length $5$ and unique Loewy filtration. As $A$ has Loewy length $3$,
any nonzero composition $A\to A\otimes_{\Bbbk}A\to A$ must map the top of $A$ to the socle of $A$.
It is easy to check that the composition of the unique (up to scalar) injection $A\hookrightarrow A\otimes_{\Bbbk}A$
and the unique (up to scalar) surjection $A\otimes_{\Bbbk}A\tto A$ is nonzero.
}
\end{remark}

Choose subalgebras $X_i$ in $Z_i$ containing $Z'_i$ and let $X=(X_1,X_2,\dots,X_n)$. Consider the
additive $2$-subcategory $\cC_{\mathcal{C},X}$ of $\cC_{\mathcal{C}}$ defined as follows: $\cC_{\mathcal{C},X}$ has the same objects and
the same $1$-morphisms as $\cC_{\mathcal{C}}$; all $2$-morphism spaces between indecomposable $1$-morphisms
in $\cC_{\mathcal{C},X}$ are the same as for $\cC_{\mathcal{C}}$ except for 
$\mathrm{End}_{\ccC_{\hspace{-2pt}\mathcal{C},X}}(\mathbbm{1}_{\mathcal{C}_i}):=X_i$.

\begin{lemma}\label{lem10}
The $2$-category  $\cC_{\mathcal{C},X}$ is well-defined and fiat. 
\end{lemma}

\begin{proof}
To prove that $\cC_{\mathcal{C},X}$ is well-defined we have to check that it is closed under both horizontal and
vertical composition of $2$-morphisms. That it is closed under vertical composition follows directly from
the fact that $X_i$ is a subalgebra. To check that it is closed under horizontal composition, we first observe that
if $\mathbbm{1}_{\mathcal{C}_i}$ appears (up to isomorphism) as a direct summand of $\mathrm{F}\circ\mathrm{G}$
for some indecomposable $1$-morphisms $\mathrm{F}$ and $\mathrm{G}$, then both $\mathrm{F}$ and 
$\mathrm{F}$ are isomorphic to $\mathbbm{1}_{\mathcal{C}_i}$. For $x,y\in X_i$, we have
\begin{displaymath}
\begin{array}{ccccccc} 
A&\overset{\sim}{\longrightarrow}&A\otimes_A A&\overset{x\otimes y}{\longrightarrow}&
A\otimes_A A&\overset{\sim}{\longrightarrow}&A\\
1&\mapsto&1\otimes 1&\mapsto&x\otimes y&\mapsto&xy
\end{array}
\end{displaymath}
from which the claim follows, again using that $X_i$ is a subalgebra.

To prove that $\cC_{\mathcal{C},X}$ is fiat we have to check that it contains all adjunction morphisms. The adjunction
morphism from $\mathbbm{1}_{\mathcal{C}_i}$ to $\mathbbm{1}_{\mathcal{C}_i}$ is $\mathrm{id}_{\mathbbm{1}_{\mathcal{C}_i}}$
and thus contained in $\cC_{\mathcal{C},X}$. All other adjunction morphisms are between $\mathbbm{1}_{\mathcal{C}_i}$ and
direct sums of indecomposable $1$-morphisms none of which is isomorphic to $\mathbbm{1}_{\mathcal{C}_i}$ and therefore
contained in $\cC_{\mathcal{C},X}$ by definition.
\end{proof}

\subsection{Description of $\mathcal{J}$-simple $2$-categories $\cC^{(\mathcal{J})}$}\label{s3.6}

Now we are ready to prove the main result of this section, which gives a description, up to
biequivalence, of fiat $2$-categories that are ``simple'' in some sense.

\begin{theorem}\label{thm11}
Let $\cC=\cC^{(\mathcal{J})}$ be a fiat $\mathcal{J}$-simple $2$-category
and assume that $\mathcal{J}$ is strongly regular. {\color{red} Assume, moreover, that Assumption \ref{bigassum} is satisfied.}
Then $\cC$ is biequivalent to $\cC_{\mathcal{C},X}$ for appropriate $\mathcal{C}$ and $X$.
\end{theorem}

\begin{proof}
Let $\mathcal{L}$ be a left cell in $\mathcal{J}$ and $\mathbf{M}:=\overline{\mathbf{C}}_{\mathcal{L}}$ be 
the corresponding cell $2$-representation. Set $\mathcal{C}_i:=\mathbf{M}(\mathtt{i})$ and let $A_i$ be a basic algebra such 
that $A_i\text{-}\mathrm{mod}$ is equivalent to  $\mathbf{M}(\mathtt{i})$. Let $Z_i$ be the center of
$A_i$ which we identify with $\mathrm{End}_{\mathfrak{R}_{\Bbbk}}(\mathbbm{1}_{\mathbf{M}(\mathtt{i})})$.
Set $X_i:=\mathbf{M}(\mathrm{End}_{\ccC}(\mathbbm{1}_{\mathtt{i}}))\subset Z_i$. 
Then the representation map $\mathbf{M}$ is a $2$-functor from $\cC$ to $\cC_{\mathcal{C},X}$, which is a biequivalence
by Theorem~\ref{thm7},  $\mathcal{J}$-simplicity of $\cC$ and construction of $X$.
\end{proof}

\section{$2$-Schur's lemma}\label{s4}

\subsection{The first layer of $2$-Schur's lemma}\label{s4.1}

Here we prove the following generalization of Theorem~\ref{thm1}.

\begin{theorem}\label{thm12}
Let $\cC$ be a fiat $2$-category and $\mathcal{J}$ a strongly regular two-sided cell of $\cC$. 
Let $\mathcal{L}$ be a left cell
of $\mathcal{J}$. {\color{red} Assume that Assumption \ref{bigassum} is satisfied for the Duflo involution in $\mathcal{L}$.} Then any endomorphism of $\mathbf{C}_{\mathcal{L}}$  is isomorphic to $\imorphism_k$ for some $k$ 
(in the category $\mathrm{End}_{\ccC\text{-}\mathrm{afmod}}(\mathbf{C}_{\mathcal{L}})$). Similarly, any endomorphism 
of $\overline{\mathbf{C}}_{\mathcal{L}}$  is isomorphic to $\imorphism_k$ for some $k$
(in the category $\mathrm{End}_{\ccC\text{-}\mathrm{mod}}(\overline{\mathbf{C}}_{\mathcal{L}})$).
\end{theorem}

\begin{proof}
We follow the proof of Theorem~\ref{thm1} described in Section~\ref{s2}.
What we need is an analogue of Lemma~\ref{lem2} in the new situation. More precise, we have to prove 
that given a non-split short exact sequence
\begin{displaymath}
0\to L_{\mathrm{G}}\to X\to L_{\mathrm{G}} \to 0
\end{displaymath}
in $\overline{\mathbf{C}}_{\mathcal{L}}(\mathtt{i})$, the obvious inclusion
$\mathrm{Ann}_{\ccC}(X)\subset \mathrm{Ann}_{\ccC}(L_{\mathrm{G}})$ is strict.

As in Subsection~\ref{s3.4}, $\overline{\mathbf{C}}_{\mathcal{L}}(\mathtt{i})$ is equivalent to 
$A\text{-}\mathrm{mod}$ for some finite dimensional associative $\Bbbk$-algebra $A$ and the functor
$\overline{\mathbf{C}}_{\mathcal{L}}(\mathrm{G})$ can be identified with tensoring with 
$Ae\otimes_{\Bbbk}eA$ for some primitive idempotent $e\in A$. By Theorem~\ref{thm7}, this identification
is fully faithful on $2$-morphisms. Clearly, 
\begin{displaymath}
\mathrm{Ann}_{\ccC}(L_{\mathrm{G}})\cap \mathrm{End}_{A\otimes_{\Bbbk}A^{\mathrm{op}}}(Ae\otimes_{\Bbbk}eA) =
eAe\otimes_{\Bbbk}\mathrm{Rad}(eAe).
\end{displaymath}
At the same time, as $X$ is a non-split self-extension of $L_{\mathrm{G}}$, we have
\begin{displaymath}
\mathrm{Ann}_{\ccC}(X)\cap \mathrm{End}_{A\otimes_{\Bbbk}A^{\mathrm{op}}}(Ae\otimes_{\Bbbk}eA) =
eAe\otimes_{\Bbbk}U,
\end{displaymath}
where $U$ is a proper subalgebra of $\mathrm{Rad}(eAe)$ (since $eA\otimes_A X=eX=X$ as a vector space).
The rest of the proof  follows precisely the proof of Theorem~\ref{thm1}.
\end{proof}

\subsection{Endomorphisms of the identity functor}\label{s5.1}

So far we have only determined the {\em objects} in the endomorphism category of a cell $2$-representation
(Theorems~\ref{thm1} and \ref{thm12}) up to isomorphism. Now we would like to describe morphisms in this category.

\begin{proposition}\label{prop14}
Let $\cC$ be a fiat $2$-category, $\mathcal{J}$ a strongly regular two-sided cell of $\cC$ and $\mathcal{L}$ a
left cell in $\mathcal{J}$. For any $k\in\mathbb{N}$, consider 
$\imorphism_k\in\mathrm{End}_{\ccC\text{-}\mathrm{mod}}(\overline{\mathbf{C}}_{\mathcal{L}})$
(or $\imorphism_k\in\mathrm{End}_{\ccC\text{-}\mathrm{mod}}({\mathbf{C}}_{\mathcal{L}})$). Then there are
isomorphisms 
\begin{displaymath}
\mathrm{End}_{\mathrm{End}_{\ccC\text{-}\mathrm{mod}}
(\overline{\mathbf{C}}_{\mathcal{L}})}(\imorphism_k)\cong\mathrm{Mat}_{k\times k}(\Bbbk)\quad\text{ and }\quad
\mathrm{End}_{\mathrm{End}_{\ccC\text{-}\mathrm{mod}}
({\mathbf{C}}_{\mathcal{L}})}(\imorphism_k)\cong\mathrm{Mat}_{k\times k}(\Bbbk). 
\end{displaymath} 
\end{proposition}

\begin{proof}
We prove the statement for $\overline{\mathbf{C}}_{\mathcal{L}}$, the other case being analogous.
For $\mathtt{i}\in\cC$, let $A_{\mathtt{i}}$ be a finite dimensional associative $\Bbbk$-algebra such that 
$\overline{\mathbf{C}}_{\mathcal{L}}(\mathtt{i})$ is equivalent to $A_{\mathtt{i}}\text{-}\mathrm{mod}$.
Let $\theta:\imorphism_k\to\imorphism_k$ be a modification. As endomorphisms of
$\mathrm{Id}_{\overline{\mathbf{C}}_{\mathcal{L}}(\mathtt{i})}$ can be identified
with the center $Z_{\mathtt{i}}$ of $A_{\mathtt{i}}$, we can view $\theta_{\mathtt{i}}$ as an element of
$\mathrm{Mat}_{k\times k}(Z_{\mathtt{i}})$.

First consider the case $k=1$. Clearly, scalars belong to the endomorphism ring of $\imorphism_1$. We would
like to show that the radical of $Z_{\mathtt{i}}$ does not. Let $e$ be a primitive idempotent of $A_{\mathtt{i}}$.
From \cite[Corollary~38(b)]{MM} it follows that there is $\mathrm{F}\in\mathcal{J}$ such that 
$\overline{\mathbf{C}}_{\mathcal{L}}(\mathrm{F})$ can be described by tensoring with a direct sum of
bimodules of the form $A_{\mathtt{i}}e\otimes_{\Bbbk}eA_{\mathtt{i}}$. The action of $\imorphism_1$ on
$\overline{\mathbf{C}}_{\mathcal{L}}(\mathtt{i})$ is described as tensoring with $A_{\mathtt{i}}$, and the isomorphism 
$\eta_{\mathrm{F}}^{\imorphism_1}$ is a direct sum of morphisms
\begin{displaymath}
A_{\mathtt{i}}\otimes_{A_{\mathtt{i}}} A_{\mathtt{i}}e\otimes_{\Bbbk}eA_{\mathtt{i}}\cong
A_{\mathtt{i}}e\otimes_{\Bbbk}eA_{\mathtt{i}}\otimes_{A_{\mathtt{i}}} A_{\mathtt{i}}
\end{displaymath}
sending $1\otimes e\otimes e$ to $e\otimes e\otimes 1$.

Let $0\neq z\in e\mathrm{Rad}(Z_{\mathtt{i}})e$. Then applying $z$ after $\eta$ sends $1\otimes e\otimes e$ to 
$e\otimes e\otimes z$, which is identified with $e\otimes z$ in $A_{\mathtt{i}}e\otimes_{\Bbbk}eA_{\mathtt{i}}$.
Applying $z$ before $\eta$ sends $1\otimes e\otimes e$ to 
$z\otimes e\otimes 1$, which is identified with $z\otimes e$ in $A_{\mathtt{i}}e\otimes_{\Bbbk}eA_{\mathtt{i}}$.
We have $e\otimes z\neq z\otimes e$ as $z\in e\mathrm{Rad}(Z_{\mathtt{i}})e$.

Now consider arbitrary $k$. From the above it follows that we can view $\theta_{\mathtt{i}}$ as an element of
$\mathrm{Mat}_{k\times k}(\Bbbk)$ (here $\Bbbk\cong Z_{\mathtt{i}}/\mathrm{Rad}(Z_{\mathtt{i}})$). That every element 
$M\in \mathrm{Mat}_{k\times k}(\Bbbk)$ indeed defines an element of $\mathrm{End}_{\mathrm{End}_{\ccC\text{-}\mathrm{mod}}
(\overline{\mathbf{C}}_{\mathcal{L}})}(\imorphism_k)$ can be seen from the commutative diagram
\begin{displaymath}
\xymatrix{
A^{\oplus k}\otimes_A Ae\otimes_{\Bbbk}eA\ar[d]^{M\otimes\mathrm{id}}\ar[rrr]^{\eta_k} &&& 
Ae\otimes_{\Bbbk}eA\otimes_AA^{\oplus k}\ar[d]^{\mathrm{id}\otimes M}\\
A^{\oplus k}\otimes_A Ae\otimes_{\Bbbk}eA\ar[rrr]^{\eta_k} &&& Ae\otimes_{\Bbbk}eA\otimes_AA^{\oplus k}\\
}
\end{displaymath}
where $A:=A_{\mathtt{i}}$ and $\eta_k$ is the diagonal $k\times k$-matrix with $\eta$ on the diagonal.
This completes the proof.
\end{proof}

\subsection{The second layer of $2$-Schur's lemma}\label{s5.2}

Our main result is the following statement.

\begin{theorem}\label{thm15}
Let $\cC$ be a fiat $2$-category, $\mathcal{J}$ a strongly regular two-sided cell of $\cC$ and $\mathcal{L}$ a
left cell in $\mathcal{J}$.  {\color{red} Assume that Assumption \ref{bigassum} is satisfied for the Duflo involution in $\mathcal{L}$.} 
Then both categories $\mathrm{End}_{\ccC\text{-}\mathrm{mod}}(\overline{\mathbf{C}}_{\mathcal{L}})$ and
$\mathrm{End}_{\ccC\text{-}\mathrm{amod}}({\mathbf{C}}_{\mathcal{L}})$ are equivalent to $\Bbbk\text{-}\mathrm{mod}$.
\end{theorem}

\begin{proof}
This follows directly from Theorems~\ref{thm1} and \ref{thm12} and Proposition~\ref{prop14}.
\end{proof}

\section{Examples}\label{s6}

\subsection{Category $\mathcal{O}$ in type $A$}\label{s6.1}

Consider the simple complex Lie algebra $\mathfrak{g}=\mathfrak{sl}_n$ with the standard triangular decomposition
$\mathfrak{g}=\mathfrak{n}_-\oplus\mathfrak{h}\oplus\mathfrak{n}_+$ and a small category $\mathcal{O}_0$ 
equivalent to the principal block of the BGG-category $\mathcal{O}$ for $\mathfrak{g}$ (see \cite{Hu}). 
Let $\cS$ be the $2$-category of projective functors associated to $\mathcal{O}_0$ as in \cite[Subsection~7.1]{MM}.
Indecomposable $1$-morphisms in  $\cS$ are in natural bijection with elements of the symmetric group
$S_n$ (the Weyl group of $\mathfrak{g}$) and left, right and two-sided cells are Kazhdan-Lusztig right, left
and two-sided cells, respectively. As shown in \cite[Subsection~7.1]{MM}, all two-sided cells are strongly regular. 
Hence Theorem~\ref{thm15} completely describes the endomorphism category of all
cell $2$-representations for $\cS$ (the latter were first constructed in \cite{MS2}). As cell $2$-representations
corresponding to the same two-sided cell are equivalent (see \cite{MS2,MM}), it follows that this equivalence
is unique (as a functor) up to isomorphism of functors. In \cite{MS2}, equivalence of cell $2$-representations
corresponding to the same two-sided cell was obtained using Arkhipov's twisting functors and the
fact that they naturally commute with projective functors, see \cite{AS}. Our present result shows that 
the shadows of Arkhipov's twisting functors act, on a cell $2$-representation, simply as a direct sum of the identity.

We also would like to note that in this example we can also apply Theorem~\ref{thm1}. A very special feature of
$S_n$ is that every two-sided Kazhdan-Lusztig cell of $S_n$ contains the longest element $w:=w_0^P$ in some 
parabolic subgroup $P$ in $S_n$. Then $w$ is the Duflo involution in its Kazhdan-Lusztig right cell and hence
the corresponding projective in the cell $2$-representation is isomorphic to $\theta_w L_w$. From 
\cite[Theorem~6.3]{MS0} it follows that the center of $\mathcal{O}_0$ surjects onto the endomorphism algebra of 
$\theta_w L_w$ and hence we can apply Theorem~\ref{thm1}.

\subsection{Category $\mathcal{O}$ in type $B_2$}\label{s6.2}

Consider the previous example for $\mathfrak{g}$ of type $B_2$. Let $W$ be the Weyl group of type $B_2$ with 
elements $\{e,s,t,st,ts,sts,tst,stst\}$ (here $s^2=t^2=e$ and $stst=tsts$). We have the $2$-category $\cS$ 
with $1$-morphisms $\theta_w$, $w\in W$. Cells are again given by Kazhdan-Lusztig combinatorics, the two-sided cells 
are $\mathcal{J}_e=\{e\}, \mathcal{J}_{s,t}=\{s,t,st,ts,sts,tst\}$ and $\mathcal{J}_{stst}=\{stst\}$. The middle 
cell splits into two left cells $\mathcal{L}_1 = \{s,st,sts\}$ and $\mathcal{L}_2= \{t,ts,tst\}$ (recall that our 
left cells are Kazhdan-Lusztig's right cells and vice versa) as shown in the following picture:
\begin{displaymath}
\begin{array}{c||c|c}
&\mathcal{L}_1&\mathcal{L}_2\\
\hline\hline
\mathcal{L}^*_1&\{s,sts\}&\{ts\}\\
\hline
\mathcal{L}^*_2&\{st\}&\{t,tst\}.
\end{array}
\end{displaymath}
Since strong regularity fails, we cannot apply Theorem~\ref{thm15} and, indeed, it turns out that the cell 
$2$-representation $\overline{\mathbf{C}}_{\mathcal{L}_1}$ has more endomorphisms than just the identity,
as we now show.

For $w\in \mathcal{L}_i$, $i=1,2$, set $L_w:=L_{\theta_w}$. Let $T_s$ and $T_t$ be Arkhipov's twisting functors 
corresponding to $s$ and $t$. Starting from $\overline{\mathbf{C}}_{\mathcal{L}_1}$ we apply $T_s$, project onto
$\overline{\mathbf{C}}_{\mathcal{L}_2}$, apply $T_t$ and project onto $\overline{\mathbf{C}}_{\mathcal{L}_1}$. 
This maps $L_s$ to $L_s\oplus L_{sts}$. As twisting functors naturally commute with projective functors, it follows
that $\mathrm{Ann}_{\ccS}(L_s)=\mathrm{Ann}_{\ccS}(L_{sts})$ and hence mapping $L_s$ to $L_{sts}$ extends to an
endomorphism of $\overline{\mathbf{C}}_{\mathcal{L}_1}$ which is clearly not isomorphic to the identity functor.

\subsection{$\mathfrak{sl}_2$-categorification}\label{s6.3}

Consider the $2$-category $\cB_n$ associated with the $\mathfrak{sl}_2$-ca\-te\-go\-ri\-fi\-ca\-tion of Chuang
and Rouquier (see \cite{CR}) as described in detail in \cite[Subsection~7.1]{MM2}. This is a fiat $2$-category
with strongly regular cells. Hence
Theorem~\ref{thm15} completely describes endomorphisms for each cell $2$-representation of $\mathfrak{sl}_2$
(compare \cite[Proposition~5.26]{CR}). However, we would like to point out that in the case of $\cB_n$ describing the
endomorphism category for cell $2$-representations is much easier (than e.g. for the example in
Subsection~\ref{s6.1}). Indeed, as explained in \cite[Subsection~7.1]{MM2}, each two-sided cell of $\cB_n$ has a 
left cell with Duflo involution $\mathrm{G}$ such that, in the corresponding cell $2$-representation, the simple 
module $L_{\mathrm{G}}$ is 
projective (the corresponding Duflo involution has the form $\mathbbm{1}_{\mathtt{i}}$). Due to this, any endomorphism
of the cell $2$-representation maps $L_{\mathrm{G}}$ to a direct sum of copies of $L_{\mathrm{G}}$ and
is uniquely determined by the image of $L_{\mathrm{G}}$ up to isomorphism.

\subsection{A non-symmetric local algebra}\label{s6.4}

In this subsection we describe an example for which the additional assumption of Theorem~\ref{thm1} 
fails, while  the conditions in Theorem~\ref{thm15} are satisfied. 
Let $A:=\Bbbk\langle x,y\rangle/(x^2,y^2,xy+yx)$ and $\mathcal{C}$ be a 
small category equivalent to $A\text{-}\mathrm{mod}$. The center $Z$ of $A$ is the linear span of $1$ and 
$xy$. Consider the fiat $2$-category $\cC_{\mathcal{C},Z}$. This category has two two-sided cells, one
consisting of the identity and the other one, say $\mathcal{J}$, consisting of the $1$-morphism $\mathrm{G}$ given 
by tensoring with $A\otimes_{\Bbbk}A$. Then $\mathrm{G}$ is the Duflo involution in $\mathcal{J}$ and the corresponding
cell $2$-representation is equivalent to the defining $2$-representation. Therefore, the projective module
$P_{\mathrm{G}}$ is isomorphic to ${}_A A$. Since $A$ is not commutative, $Z$ does not surject on the
endomorphism algebra of $P_{\mathrm{G}}$. Hence the additional assumption of Theorem~\ref{thm1} is not satisfied.
On the other hand, the conditions in Theorem~\ref{thm15} are satisfied as explained in 
\cite[Subsection~7.3]{MM}.

\section{Graded fiat $2$-categories}\label{s7}

In the original version of the paper, the main results of this paper were stated under an 
additional numerical assumption which was shown to be redundant in \cite{MM5}. 
The original version of this section contained an argument that the
numerical assumption is satisfied for graded fiat $2$-categories.
Although the result itself is no longer interesting, the setup of graded fiat $2$-categories is
of interest (as most of the natural examples of fiat $2$-categories are graded)
and this is what is presented in this section, leading up to an analogue of 
Lusztig's $\mathbf{a}$-function for graded fiat $2$-categories.

In this section, by {\em graded} we always mean $\mathbb{Z}$-graded.

\subsection{$2$-categories with free $\mathbb{Z}$-action}\label{s7.1}

Let $\cA$ be $2$-category. Assume that, for each $\mathtt{i},\mathtt{j}\in\cA$, we are given an automorphism 
$(\cdot)_1$ of $\cA(\mathtt{i},\mathtt{j})$. For $k\in\mathbb{Z}$, set $(\cdot)_k:=(\cdot)_1^{k}$ and,
for $\mathrm{F}\in \cA(\mathtt{i},\mathtt{j})$, set $\mathrm{F}_k:=(\mathrm{F})_k$.
We will say that this datum defines a {\em free} action of $\mathbb{Z}$ on $\cA$ provided that, 
for any $\mathrm{F}\in \cA(\mathtt{i},\mathtt{j})$, the equality $\mathrm{F}_k=\mathrm{F}_m$ implies
$k=m$ and, moreover, for any composable $1$-morphisms $\mathrm{F}$ and $\mathrm{G}$, we have
\begin{equation}\label{eq19}
\mathrm{F}_k\circ\mathrm{G}_m=(\mathrm{F}\circ\mathrm{G})_{k+m}.
\end{equation}

\begin{example}\label{ex16}
{  
Let $A$ be a graded, connected, weakly symmetric finite dimensional associative $\Bbbk$-algebra
and $\mathcal{C}$ a small category equivalent to the category $A\text{-}\mathrm{gmod}$ of finite
dimensional graded $A$-modules. The algebra $A\otimes_{\Bbbk}A^{\mathrm{op}}$ inherits the structure of a 
graded algebra from $A$. Let $\langle 1\rangle$ denote the functor which shifts the grading
such that $(M\langle 1\rangle)_i=M_{i+1}$, $i\in\mathbb{Z}$. Consider the $2$-category $\cC_{\mathcal{C}}$
defined as follows: It has one object (which we identify with $\mathcal{C}$), its $1$-morphisms are closed
under isomorphism of functors and are generated by $\langle \pm 1\rangle$ and functors induced by tensoring
with projective $A\text{-}A$-bimodules (the latter are naturally graded), its $2$-morphisms are natural 
transformations of functors (which correspond to homogeneous bimodule morphisms of degree zero).
The group $\mathbb{Z}$ acts on $\cC_{\mathcal{C}}$ by shifting the grading and this is free in the above sense.
}
\end{example}

\subsection{Graded fiat $2$-categories}\label{s7.2}

Assume that $\cA$ is a $2$-category equipped with a free action of $\mathbb{Z}$. Assume further that $\cA$
satisfies the following conditions:
\begin{itemize}
\item  $\cA$ has finitely many objects;
\item for any $\mathtt{i},\mathtt{j}\in\cA$, we have $\cA(\mathtt{i},\mathtt{j})\in \mathfrak{A}_{\Bbbk}$
and horizontal composition is both additive and $\Bbbk$-linear;
\item the set of $\mathbb{Z}$-orbits on isomorphism classes of indecomposable objects in 
$\cA(\mathtt{i},\mathtt{j})$ is finite;
\item all spaces of $2$-morphisms are finite dimensional;
\item for each $1$-morphism $\mathrm{F}$, there are only finitely many indecomposable $1$-mor\-phisms 
$\mathrm{G}$ (up to isomorphism) such that $\mathrm{Hom}_{\ccA}(\mathrm{F},\mathrm{G})\neq 0$;
\item for each $1$-morphism $\mathrm{F}$, there are only finitely many indecomposable $1$-mor\-phisms 
$\mathrm{G}$ (up to isomorphism) such that $\mathrm{Hom}_{\ccA}(\mathrm{G},\mathrm{F})\neq 0$;
\item for any $\mathtt{i}\in\cC$ the $1$-morphism $\mathbbm{1}_{\mathtt{i}}$ is indecomposable;
\item $\cA$ has a weak object preserving involution and adjunction morphisms.
\end{itemize}
We will call such $\cA$ {\em pro-fiat}.

Define the quotient $2$-category $\cC=\cA/\mathbb{Z}$ to have the same objects as $\cA$, and as 
morphism categories the categorical quotients $\cC(\mathtt{i},\mathtt{j}):=\cA(\mathtt{i},\mathtt{j})/\mathbb{Z}$.
Recall that objects of $\cA(\mathtt{i},\mathtt{j})/\mathbb{Z}$ are orbits of $\mathbb{Z}$ acting 
on objects of $\cA(\mathtt{i},\mathtt{j})$ (for $\mathrm{F}\in \cA(\mathtt{i},\mathtt{j})$, we will denote 
the corresponding orbit by $\mathrm{F}_{\bullet}$) and, for $\mathrm{F},\mathrm{G}\in \cA(\mathtt{i},\mathtt{j})$, the 
space $\mathrm{Hom}_{\ccC}(\mathrm{F}_{\bullet},\mathrm{G}_{\bullet})$
is the quotient of  $\bigoplus_{k,l\in\mathbb{Z}}\mathrm{Hom}_{\ccA(\mathtt{i},\mathtt{j})}
(\mathrm{F}_k,\mathrm{G}_l)$ modulo the subspace generated by the expressions
$\alpha-\alpha_l$ for $l\in\mathbb{Z}$. Horizontal composition in $\cC$ is induced by the one in $\cA$ in the natural way
(which is well-defined due to \eqref{eq19}). We denote by $\Omega:\cA\to\cC$ the projection $2$-functor.

Thanks to our assumptions on $\cA$, the $2$-category $\cC$ is a fiat $2$-category. We will say that $\cC$ is a 
{\em graded fiat} $2$-category. If we fix a representative $\mathrm{F}_s$ in each $\mathrm{F}_{\bullet}$, then, 
by construction, the category $\cC(\mathtt{i},\mathtt{j})$ becomes graded (in the sense that for any 
$1$-morphisms $\mathrm{F}_{\bullet},\mathrm{G}_{\bullet}$ we have 
\begin{displaymath}
\mathrm{Hom}_{\ccC}(\mathrm{F}_{\bullet},\mathrm{G}_{\bullet})=
\bigoplus_{i\in\mathbb{Z}}\mathrm{Hom}^{i}_{\ccC}(\mathrm{F}_{\bullet},\mathrm{G}_{\bullet}), 
\end{displaymath}
where $\mathrm{G}_t$ is our fixed representative for $\mathrm{G}_{\bullet}$ and
$\mathrm{Hom}^{i}_{\ccC}(\mathrm{F}_{\bullet},\mathrm{G}_{\bullet})=\mathrm{Hom}_{\ccA}(\mathrm{F}_s,\mathrm{G}_{t+i})$,
vertical composition being additive on degrees). We will say that this grading is {\em positive} provided that 
the following condition is satisfied: for any indecomposable $1$-morphisms 
$\mathrm{F}_{\bullet},\mathrm{G}_{\bullet}\in\cC$, the inequality 
$\mathrm{Hom}^{i}_{\ccC}(\mathrm{F}_{\bullet},\mathrm{G}_{\bullet})\neq 0$ implies $i>0$ unless 
$\mathrm{F}_{\bullet}=\mathrm{G}_{\bullet}$. In the latter case we require
$\mathrm{End}^{0}_{\ccC}(\mathrm{F}_{\bullet})=\Bbbk\,\mathrm{id}_{\mathrm{F}_{\bullet}}$. 

\begin{example}\label{ex17}
{  
Let $D=\Bbbk[x]/(x^2)$ with $x$ in degree $2$ and consider $\cC_{\mathcal{C}}$ as in Example~\ref{ex16}
for some $\mathcal{C}$ equivalent to $D\text{-}\mathrm{gmod}$. Choosing the representatives $\mathrm{Id}_{D\text{-}\mathrm{gmod}}$ and
$(D\otimes_{\Bbbk}D\otimes_D{}_-)\langle 1\rangle$ makes $\cC_{\mathcal{C}}/\mathbb{Z}$ into a positively 
graded $2$-category.
}
\end{example}

\subsection{From $2$-representations of $\cA$ to $2$-representations of $\cC$}\label{s7.3}

Let $\cA$ be a pro-fiat $2$-category and $\cC:=\cA/\mathbb{Z}$.
Let $\mathbf{M}$ be a $2$-representation of $\cA$ and $\mathtt{i}\in\cA$. Then 
the group $\mathbb{Z}$ acts (strictly) on $\mathbf{M}(\mathtt{i})$ via isomorphisms $\mathbbm{1}_{\mathtt{i},k}$,
$k\in \mathbb{Z}$. We call $\mathbf{M}$ {\em pro-graded} if this action is free (i.e. the stabilizer of
every object is trivial) for every $\mathtt{i}$.

Let $\mathbf{M}$ be a pro-graded $2$-representation of $\cA$. We define a $2$-representation
$\underline{\mathbf{M}}$ of $\cC$ as follows: For $\mathtt{i}\in\cC$, we set 
$\underline{\mathbf{M}}(\mathtt{i}):=\mathbf{M}(\mathtt{i})/\mathbb{Z}$, that is objects of 
$\underline{\mathbf{M}}(\mathtt{i})$ are orbits of $\mathbb{Z}$ acting on objects of $\mathbf{M}(\mathtt{i})$
(for $Q\in \mathbf{M}(\mathtt{i})$, we will denote the corresponding orbit by $(Q)$).
For $\mathrm{F}\in\cA(\mathtt{i},\mathtt{j})$ and
$Q\in \mathbf{M}(\mathtt{i})$, we define $\underline{\mathbf{M}}(\mathrm{F}_{\bullet})\, 
(Q):=({\mathbf{M}}(\mathrm{F})\, Q)$ while, for $f:Q\to P$, mapping the class $\hat{f}:(Q)\to (P)$ to the class 
\begin{displaymath}
\widehat{{\mathbf{M}}(\mathrm{F})f}:({\mathbf{M}}(\mathrm{F})\,Q)\to ({\mathbf{M}}(\mathrm{F})\,P) 
\end{displaymath}
defines the action of $\underline{\mathbf{M}}(\mathrm{F}_{\bullet})$ on morphisms (this is well-defined because of
the strictness of our $\mathbb{Z}$-action). Functoriality of $\underline{\mathbf{M}}(\mathrm{F}_{\bullet})$ 
follows directly from the definition. Each $\alpha:\mathrm{F}\to \mathrm{G}$ induces a 
morphism from $\mathrm{F}_{\bullet}$ to $\mathrm{G}_{\bullet}$ and we define 
\begin{displaymath}
\underline{\mathbf{M}}(\alpha)_{(Q)}:\underline{\mathbf{M}}(\mathrm{F}_{\bullet})\, (Q)\to 
\underline{\mathbf{M}}(\mathrm{G}_{\bullet})\, (Q) 
\end{displaymath}
as the class of ${\mathbf{M}}(\alpha)_{Q}:{\mathbf{M}}(\mathrm{F})\, Q\to {\mathbf{M}}(\mathrm{G})\, Q$.
This extends to all $2$-morphisms by additivity. It follows directly from the definitions that 
$\underline{\mathbf{M}}$ becomes a $2$-\-re\-pre\-sen\-ta\-ti\-on of $\cC$.

\subsection{Functoriality of $\underline{\,\cdot\,}$}\label{s7.9}

Unfortunately, $\underline{\,\cdot\,}$ is not a $2$-functor between the $2$-categories of $2$-representations of
$\cA$ and $\cC=\cA/\mathbb{Z}$. However, it turns out to be a $2$-functor on a suitably defined subcategory
of $2$-representations of $\cA$. Define the $2$-category $\cA\text{-}\mathrm{pgamod}$ as follows: 
objects are pro-graded additive $2$-representations of $\cA$; $1$-morphisms are $2$-natural transformations 
satisfying the condition that $\eta_{\mathbbm{1}_{\mathtt{i},n}}$ is the identity map for all $\mathtt{i}$ and 
$n$ (that is, our $2$-natural transformations commute {\em strictly} with all shifts of the identity); $2$-morphisms
are modifications. This clearly forms a $2$-subcategory in the category of additive $2$-representations of $\cA$. 

\begin{proposition}\label{prop47}
The operation $\underline{\,\cdot\,}$ defines a $2$-functor from $\cA\text{-}\mathrm{pgamod}$
to $\cC\text{-}\mathrm{amod}$.
\end{proposition}

\begin{proof}
Let $\mathbf{M},\mathbf{N}\in\cA\text{-}\mathrm{pgamod}$ and $\Psi\in\mathrm{Hom}_{\cA\text{-}\mathrm{pgamod}}
(\mathbf{M},\mathbf{N})$. Define $\underline{\Psi}:\underline{\mathbf{M}}\to\underline{\mathbf{N}}$ by
$\underline{\Psi}_{\mathtt{i}}\, (Q):=(\Psi_{\mathtt{i}}\, Q)$. This is well defined as $\Psi_{\mathtt{i}}$
commutes strictly with the action of $\mathbbm{1}_{\mathtt{i},n}$ and each element in $(Q)$ is obtained by
applying some $\mathbbm{1}_{\mathtt{i},n}$ to $Q$. We have to check commutativity of the  diagram
\begin{displaymath}
\xymatrix{
\underline{\Psi}_{\mathtt{j}}\circ \underline{\mathbf{M}}(\mathrm{F}_{\bullet})\ar[rr]^{\eta_{\mathrm{F}_{\bullet}}}
\ar[d]_{\mathrm{id}_{\underline{\Psi}_{\mathtt{j}}}\circ_0 \underline{\mathbf{M}}(\alpha)}&&
\underline{\mathbf{N}}(\mathrm{F}_{\bullet})\circ\underline{\Psi}_{\mathtt{i}}
\ar[d]^{\underline{\mathbf{N}}(\alpha)\circ_0 \mathrm{id}_{\underline{\Psi}_{\mathtt{i}}}}\\
\underline{\Psi}_{\mathtt{j}}\circ \underline{\mathbf{M}}(\mathrm{G}_{\bullet})\ar[rr]^{\eta_{\mathrm{G}_{\bullet}}}&&
\underline{\mathbf{N}}(\mathrm{G}_{\bullet})\circ\underline{\Psi}_{\mathtt{i}}\\
} 
\end{displaymath}
for any $\alpha:\mathrm{F}\to\mathrm{G}$ in $\cA$ (here $\eta_{\mathrm{F}_{\bullet}}$ is the class of 
$\eta_{\mathrm{F}}$ and similarly for $\eta_{\mathrm{G}_{\bullet}}$). To check commutativity of this diagram, we have 
to evaluate it at any object and it is straightforward to check commutativity there using strict commutativity of
$\Psi$ with shifts of the identity. Condition \eqref{eq3} for $\eta_{\mathrm{F}_{\bullet}}$ is automatic.
This verifies the first level of $2$-functoriality.

For a modification $\theta:\Psi\to\Phi$ in $\cA\text{-}\mathrm{pgamod}$, we define $\underline{\theta}$ by
$\underline{\theta}_{\mathtt{i},(Q)}:=\widehat{\theta_{\mathtt{i},Q}}$. We have to check \eqref{eq33}, that is
commutativity of the diagram
\begin{displaymath}
\xymatrix{ 
\underline{\Psi}_{\mathtt{j}}\circ \underline{\mathbf{M}}(\mathrm{F}_{\bullet})
\ar[rr]^{\eta_{\mathrm{F}_{\bullet}}^{\Psi}}
\ar[d]_{\underline{\theta}_{\mathtt{j}}\circ_0 \underline{\mathbf{M}}(\alpha)}&&
\underline{\mathbf{N}}(\mathrm{F}_{\bullet})\circ\underline{\Psi}_{\mathtt{i}}
\ar[d]^{\underline{\mathbf{N}}(\alpha)\circ_0 \underline{\theta}_{\mathtt{i}}}\\
\underline{\Phi}_{\mathtt{j}}\circ \underline{\mathbf{M}}(\mathrm{G}_{\bullet})
\ar[rr]^{\eta_{\mathrm{G}_{\bullet}}^{\Phi}}&&
\underline{\mathbf{N}}(\mathrm{G}_{\bullet})\circ\underline{\Phi}_{\mathtt{i}}\\
}
\end{displaymath}
which again follows by evaluating it at any object and using strict commutativity of
$\Psi$ and $\Phi$ with shifts of the identity.
\end{proof}

\subsection{Principal and cell $2$-representations of $\cA$}\label{s7.4}

For $\mathtt{i}\in \cA$, consider the principal $2$-representation $\mathbf{P}_{\mathtt{i}}^{\ccA}$ of
$\cA$.

\begin{proposition}\label{prop21}
The $2$-representations $\underline{\mathbf{P}_{\mathtt{i}}^{\ccA}}$ and 
$\mathbf{P}_{\mathtt{i}}$ of $\cC$ are equivalent.
\end{proposition}

\begin{proof}
First we note that $\mathbf{P}_{\mathtt{i}}^{\ccA}$ is pro-graded by definition.
For $\mathtt{j}\in\cC$, the orbits of $\mathbb{Z}$ on $\underline{\mathbf{P}_{\mathtt{i}}^{\ccA}}(\mathtt{j})$
coincide with the fibers of $\Omega$ on $\cC(\mathtt{i},\mathtt{j})$. The equivalence is then defined 
by mapping the fiber to its image under $\Omega$.
\end{proof}

Directly from the definitions, we have that $\underline{(\overline{\mathbf{M}})}=
\overline{(\underline{\mathbf{M}})}$ for any $2$-representation $\mathbf{M}$ of $\cA$. Consider
the $2$-representation $\overline{\mathbf{P}_{\mathtt{i}}^{\ccA}}$. By definition, each
$\overline{\mathbf{P}_{\mathtt{i}}^{\ccA}}(\mathtt{j})$ is a length category with enough projective objects.
For any $\mathtt{j}$, there is a bijection between isomorphism classes of simple objects in 
$\overline{\mathbf{P}_{\mathtt{i}}}(\mathtt{j})$ and $\mathbb{Z}$-orbits on isomorphism classes of simple objects in 
$\overline{\mathbf{P}_{\mathtt{i}}^{\ccA}}(\mathtt{j})$. 

The $2$-functor $\Omega$ induces a bijection between left, right and two-sided cells of $\cA$ and $\cC$.
Let $\mathcal{L}$ be a left cell in $\cC$ and $\mathrm{G}$ a $1$-morphism in $\cA$ such that 
$\mathrm{G}_{\bullet}$ is the Duflo involution in $\mathcal{L}$.  Setting $Q:=\mathrm{G}\, L_{\mathrm{G}}$
as in Subsection~\ref{s1.6}, we consider the $2$-representation $\mathbf{C}_{\mathcal{L}}^{\ccA}:=
(\overline{\mathbf{P}_{\mathtt{i}}^{\ccA}}(\mathtt{j}))_{Q}$. We leave it to the reader to check that this
is the cell $2$-representation of $\cA$ associated with $\Omega^{-1}(\mathcal{L})$.

\begin{proposition}\label{prop22}
The $2$-representations $\underline{\mathbf{C}_{\mathcal{L}}^{\ccA}}$ and 
$\mathbf{C}_{\mathcal{L}}$ of $\cC$ are equivalent.
\end{proposition}

\begin{proof}
The fact that $\mathbf{C}_{\mathcal{L}}^{\ccA}$ is pro-graded follows from the definition of
$\mathbf{C}_{\mathcal{L}}^{\ccA}$ and the fact that $\mathbf{P}_{\mathtt{i}}^{\ccA}$ is pro-graded.
Similarly to Proposition~\ref{prop21}, the equivalence is induced by $\Omega$.
\end{proof}

\subsection{Graded adjunctions}\label{s7.5}

Let $\cA$ be a pro-fiat $2$-category and $\cC:=\cA/\mathbb{Z}$. Let $\mathcal{L}$ be a strongly 
regular left cell of $\cC$ and $\mathtt{i}:=\mathtt{i}_{\mathcal{L}}$. We assume that we have chosen some
representatives in $\mathbb{Z}$-orbits such that the induced grading on $\cC$ is positive. 
We also assume that $\mathbbm{1}_{\mathtt{i},\bullet}$ is represented by 
the identity $1$-morphism $\mathbbm{1}_{\mathtt{i},0}$  in $\cA(\mathtt{i},\mathtt{i})$.
Let $\mathrm{G}_{\bullet}$ be the Duflo involution for $\mathcal{L}$ and let $\mathrm{G}$ be its
chosen representative in $\cA(\mathtt{i},\mathtt{i})$. 

We have $\mathrm{Hom}_{\ccC}(\mathrm{G}_{\bullet},\mathbbm{1}_{\mathtt{i},\bullet})\neq 0$ by
\cite[Proposition~17]{MM} and hence it makes sense to define
$\mathbf{a}$ as the smallest integer such that 
\begin{displaymath}
\mathrm{Hom}^{\mathbf{a}}_{\ccC}(\mathrm{G}_{\bullet},\mathbbm{1}_{\mathtt{i},\bullet})=
\mathrm{Hom}_{\ccA}(\mathrm{G}_{-\mathbf{a}},\mathbbm{1}_{\mathtt{i},0})\neq 0. 
\end{displaymath}
This should be thought of as an analogue of Lusztig's $\mathbf{a}$-function.

Consider the cell $2$-representation $\mathbf{C}_{\mathcal{L}}$ of $\cC$. By Proposition~\ref{prop22},
we have a positive grading on $\mathbf{C}_{\mathcal{L}}(\mathtt{i})$. Denote by $\mathbf{l}$ the maximal
$i\in\mathbb{Z}$ such that $\mathrm{End}^i(P_{\mathrm{G}_{\bullet}})\neq 0$.

\begin{lemma}\label{lem32}
We have $\mathrm{G}^*\cong \mathrm{G}_{\mathbf{l}-2\mathbf{a}}$.
\end{lemma}

\begin{proof}
As $\mathrm{G}_{\bullet}^*\cong \mathrm{G}_{\bullet}$, we have $\mathrm{G}^*\cong \mathrm{G}_x$ for some
$x\in \mathbb{Z}$. As in \cite[Subsection~4.7]{MM}, we denote by $\Delta$ the unique quotient of 
$0\to \mathbbm{1}_{\mathtt{i},0}$ which has simple socle $L_{\mathrm{G}_{-\mathbf{a}}}$. We compute:
\begin{displaymath}
\begin{array}{rcl}
0&\neq & \mathrm{Hom}(\mathrm{G}\,\mathbbm{1}_{\mathtt{i},0},L_{\mathrm{G}})\\
&\subset& \mathrm{Hom}(\mathrm{G}\,\Delta,L_{\mathrm{G}})\\
&=& \mathrm{Hom}(\mathrm{G}\,L_{\mathrm{G}_{-\mathbf{a}}},L_{\mathrm{G}})\\
&=& \mathrm{Hom}(L_{\mathrm{G}_{-\mathbf{a}}},\mathrm{G}_x\,L_{\mathrm{G}})\\
&=& \mathrm{Hom}(L_{\mathrm{G}_{-\mathbf{a}}},\mathrm{G}_{x+\mathbf{a}}\,L_{\mathrm{G}_{-\mathbf{a}}}).
\end{array}
\end{displaymath}
Here the third line follows from the fact that $\mathrm{G}$ annihilates all subquotients of $\Delta$ apart from
$L_{\mathrm{G}_{-\mathbf{a}}}$ (see \cite[Proposition~17]{MM}), and the fourth line uses adjunction.
The module $\mathrm{G}_{x+\mathbf{a}}\,L_{\mathrm{G}_{-\mathbf{a}}}$ has simple socle
$L_{\mathrm{G}_{x+\mathbf{a}-\mathbf{l}}}$. Therefore, the inequality
$\mathrm{Hom}(L_{\mathrm{G}_{-\mathbf{a}}},\mathrm{G}_{x+\mathbf{a}}\,L_{\mathrm{G}_{-\mathbf{a}}})\neq 0$
means that $-\mathbf{a}=x+\mathbf{a}-\mathbf{l}$, that is $x=\mathbf{l}-2\mathbf{a}$.
\end{proof}

\section{Corrigendum to the proof of Theorem \ref{thm7}}\label{corrig}

As stated above, Theorem \ref{thm7} is wrong in general. A counterexample is given by $\cC=\mathrm{Rep}(C_2)$ over a field $\Bbbk$ of characteristic $2$, where $\mathrm{Rep}(C_2)$ denotes the category of representations of the cyclic group of order $2$ with horizontal composition given by the tensor product. This is a $2$-category with one object $\bullet$, two indecomposable $1$-morphisms: the simple module $S$ which acts as the identity $1$-morphism and the projective module $P$. There are two $\mathcal{J}$-cells $\{S\}$ and $\mathcal{J}=\{P\}$. The $2$-category is $\mathcal{J}$-simple and the cell $2$-representation $\mathbf{C}_{\mathcal{J}}$ has underlying algebra isomorphic to $A=\Bbbk[x]/(x^2)$. The functor $\mathbf{C}_{\mathcal{J}}(P)$ if given by tensoring with $A\otimes A$. The endomorphism algebra of $\mathrm{G}=P$ is isomorphic to $A$ and the representation map $A\to A\otimes_\Bbbk A$ is given by $x\mapsto x\otimes 1+1\otimes x$, which is an algebra morphism in characteristic $2$. This map is clearly not surjective. 

The problem here stems from the fact that the composition of the (unique up to scalar) morphisms $S\to P\to S$ is zero, which under the representation map translates to the composition $A\to A\otimes_\Bbbk A \to A$, where the first map is the one above and the second is given by multiplication, being zero.

In order to prove Theorem \ref{thm7} under the additional assumption that the composition 
$\mathbbm{1}_{\mathtt{i}} \to \mathrm{G} \to \mathbbm{1}_{\mathtt{i}} $ of the map defining the Duflo involution with its mate under the adjunction isomorphism
\[ \mathrm{Hom}_{\ccC}(\mathrm{G},\mathbbm{1}_{\mathtt{i}})\cong  \mathrm{Hom}_{\ccC}(\mathbbm{1}_{\mathtt{i}},\mathrm{G})\] 
is nonzero, we first need a general result about finite-dimensional algebras.

\subsection{Subalgebras of self-injective algebras}

Let $\Bbbk$ be an algebraically closed  
field and $A$ a finite dimensional, local
and self-injective associative $\Bbbk$-algebra.
Let $R$ be the radical of $A$. Let
$s$ be a fixed non-zero element in the 
(one-dimensional) socle of $A$.

Consider the enveloping algebra $U:=A\otimes_\Bbbk A^{\mathrm{op}}$.
Note that $U$ is also self-injective and local.

\begin{proposition}\label{mainprop}
Let $Q$ be a subalgebra of $U$ satisfying the following 
conditions.
\begin{enumerate}[$($a$)$]
\item\label{mainprop.1} $Q$ is self-injective.
\item\label{mainprop.2} $Q$ contains both $1\otimes s$
and $s\otimes 1$.
\item\label{mainprop.3} For any $x\in A$, there is 
$u_x\in A\otimes_\Bbbk R$ such that $Q$ contains
$(x\otimes 1) + u_x$.
\item\label{mainprop.4} For any $x\in A$, there is 
$v_x\in R\otimes_\Bbbk A$ such that $Q$ contains
$(1\otimes x) + v_x$.
\end{enumerate}
Then $Q=U$. 
\end{proposition}

\begin{proof}
As $Q$, by Assumption~\eqref{mainprop.2}, contains  both $1\otimes s$
and $s\otimes 1$, it contains their product $s\otimes s$
which is a generator of the simple socle of $U$.
In particular, $Q$ contains the socle of $U$ which then
is a part of the socle of $Q$. Therefore the socle of
$Q$ must coincide with the socle of $U$ due to our Assumption~\eqref{mainprop.1}
that $Q$ is self-injective. 
Our strategy of the proof will
be to seek the following contradiction: assume $Q\neq U$ 
and show that in this case $Q$ has an additional socle component.

Choose $a_1,\dots,a_n$ in $R$
which descend to a basis of $R/R^2$ modulo $R^2$. Then
$a_1,\dots,a_n$ generate $A$. Moreover, 
\begin{displaymath}
a_1\otimes 1,\dots,a_n\otimes 1,1\otimes a_1,\dots,1\otimes a_n 
\end{displaymath}
generate $U$.
By Assumptions~\eqref{mainprop.3} and \eqref{mainprop.4},
we know that $Q$ contains $(a_i\otimes 1) + u_{a_i}$
and $(1\otimes a_i) + v_{a_i}$, for $i=1,2,\dots,n$.
Note that that, in principle, some 
$(a_i\otimes 1) + u_{a_i}$ can coincide with some 
$(1\otimes a_j) + v_{a_j}$.

As $A$ is self-injective and local, we know that 
the socle $\mathrm{soc}(A)$ of $A$ is generated by $s$.
Let $b_1,\dots ,b_n$ be elements of $\mathrm{soc}^2(A)$
which descend to a basis of $\mathrm{soc}^2(A)/\mathrm{soc}(A)$
modulo $\mathrm{soc}(A)$. Note that $n$ is the same as in the previous
paragraph which is justified by the fact that $A$ is self-injective.
Then the elements
\begin{displaymath}
b_1\otimes s,\dots b_n\otimes s,
s\otimes b_1,\dots, s\otimes b_n
\end{displaymath}
belong to $\mathrm{soc}^2(U)$ and descend to a basis of 
$\mathrm{soc}^2(U)/\mathrm{soc}(U)$
modulo $\mathrm{soc}(U)$.

For $i\in\{1,\dots,n\}$, let $w_i$ be an element of the
free algebra $\mathcal{F}$ with generators $x_1,\dots,x_n$
which descends to $b_i$ under the canonical projection
$\mathcal{F}\tto A, x_j \mapsto a_j$. Now let 
$\phi\colon \mathcal{F} \to U$ be the map defined by $\phi(x_j) = (a_j\otimes 1)+u_{a_j}$ and let $\overline{w}_i = \phi(w_i)$, which is in $Q$. Since $u_{a_j}\in A\otimes R$ by assumption, we have $(1\otimes s)u_{a_j}=0$. 
Hence $(1\otimes s)\overline{w}_i= b_i\otimes s \in Q$. Similarly, we obtain $s\otimes b_i \in Q$.
Consequently, $Q$ contains $\mathrm{soc}^2(U)$.
In particular, the space $\mathrm{soc}^2(U)/\mathrm{soc}(U)$
is a subquotient of $Q$ and has dimension $2n$.


Let us now assume that $Q\neq U$. Then the dimension of $M:=\mathrm{rad}(Q)/ (Q\cap \mathrm{rad}^2(U))$ is
strictly less that $2n$ as $\mathrm{rad}(U)/  \mathrm{rad}^2(U)$ has dimension $2n$
and generates $U$. We denote  the dimension of $M$ by $l$
and let $m_1,\dots,m_l$ be elements 
in $\mathrm{rad}(Q)$ which descend to a basis of $M$
under the canonical projection. For each $i\in\{1,\dots,l\}$, multiplication by $m_i$ defines a linear map from
the $2n$-dimensional space $\mathrm{soc}^2(U)/\mathrm{soc}(U)$
to the $1$-dimensional space $\mathrm{soc}(U)$
and this linear map has kernel of dimension at least $2n-1$.

Since the number of $m_i$ is strictly smaller that $2n$, we necessarily will
find a non-zero element of $\mathrm{soc}^2(U)/\mathrm{soc}(U)$
which is mapped to $0$ by all of them, and hence by the radical of $Q$. This means 
exactly that $\mathrm{soc}^2(U)$ contains at least 
two linearly independent elements in the socle of $Q$.
This is our contradiction.
\end{proof}

\subsection{New version of Theorem \ref{thm7}}

Now we place ourselves into the context of Theorem \ref{thm7}. That is, we let $\cC= \cC^{(\mathcal{J})}$ be a fiat $2$-category such that $\mathcal{J}$ is strongly regular. We let $\mathcal{L}$ be a left cell in $\mathcal{J}$ and $\mathrm{G}$ its Duflo involution with source and target $\mathtt{i}$.

We first note that we can consider the $\mathcal{H}$-cell reduction $\cC_{\mathcal{H}}$ of $\cC$, which is the $2$-category with one object $\mathtt{i}$ and morphism category $\cC_{\mathcal{H}}(\mathtt{i},\mathtt{i})$ given by the additive closure of $\mathbbm{1}_{\mathtt{i}}$ and $\mathrm{G}$ in $\cC(\mathtt{i},\mathtt{i})$. 

\begin{lemma}
The cell $2$-representation $\mathbf{C}_{\mathcal{L}}$ of $\cC$ is $\mathcal{J}$-$2$-full if and only if the cell $2$-representation $\mathbf{C}_{\mathcal{H}}$ of $\cC_{\mathcal{H}}$ is $\mathcal{H}$-$2$-full.
\end{lemma}

\begin{proof}
This follows directly from Propositions~\ref{prop4} and \ref{prop6}.
\end{proof}

It thus suffices to prove the theorem in the special case
of the category $\cC_{\mathcal{H}}$. In other words,
we assume that $\cC=\cC_{\mathcal{H}}$ is a fiat $2$-category 
(with one object) and two $1$-morphisms
$\mathrm{G}$ and $\mathbbm{1}$, where $\mathcal{J}=\{\mathrm{G}\}$ is a
two-sided cell (hence also a left cell, which we denote $\mathcal{L}$). 
We assume $\cC$ is $\mathcal{J}$-simple.
Consider the cell $2$-representation $\mathbf{C}_{\mathcal{L}}$
and its abelianization $\overline{\mathbf{C}}_{\mathcal{L}}$.
We denote the (unique, up to isomorphism) indecomposable projective
in the underlying category of this $2$-representation by $P_{\mathrm{G}}$
and its the simple top by $L_\mathrm{G}$. We denote by $A$
the endomorphism algebra of $P_\mathrm{G}$ and by $R$ the radical of $A$.
Note that $A$ is local (as $P_\mathrm{G}$ is indecomposable).
Additionally, from \cite[Theorem~2]{KMMZ} 
(and the dual dual statement for injective modules)
it follows that $A$ is  self-injective.

Then the  representation map $\mathbf{C}_{\mathcal{L}}$
maps ${\cC}$ to the bicategory $A$-mod-$A$ of 
all finite dimensional $A$-$A$-bimodules. The image of
$\mathbbm{1}$ under $\mathbf{C}_{\mathcal{L}}$ is 
isomorphic to the regular $A$-$A$-bimodule $A$, while
the image of  $\mathrm{G}$ under $\mathbf{C}_{\mathcal{L}}$ is 
isomorphic to the indecomposable projective 
$A$-$A$-bimodule $A\otimes_\Bbbk A$.
As $\mathrm{G}$ is not annihilated by $\mathbf{C}_{\mathcal{L}}$
and ${\cC}$ is $\mathcal{J}$-simple, the 
representation map $\mathbf{C}_{\mathcal{L}}$ from ${\cC}$
to $A$-mod-$A$ is injective at the level of $2$-morphisms.

Recall that ${\cC}_A$ denote the sub-bicategory of 
$A$-mod-$A$ given by the additive closure of $A$
and $A\otimes_\Bbbk A$. By the above, the 
representation map $\mathbf{C}_{\mathcal{L}}$ maps 
${\cC}$ to ${\cC}_A$.

\begin{theorem}\label{thmmain}
\begin{enumerate}[$($a$)$]
\item\label{thmmain.1}
The composition $\alpha\circ\overline{\alpha}$ is non-zero
if and only if $\mathrm{char}(\Bbbk)$ does not divide 
the dimension of $A$. 
\item\label{thmmain.2}
If the latter condition is satisfied,
then  the representation map $\mathbf{C}_{\mathcal{L}}$ from
${\cC}$ to ${\cC}_A$ is a biequivalence.
\end{enumerate}
\end{theorem}
 
\begin{proof}
We need to prove surjectivity of  
the representation map $\mathbf{C}_{\mathcal{L}}$ from
${\cC}$ to ${\cC}_A$ at the level of 
two morphisms. Similarly to 
Propositions~\ref{prop4} and \ref{prop6},
by adjunction, this reduces to the surjectivity
of 
\begin{equation}\label{eq-ad1}
{\cC}(\mathrm{G},\mathbbm{1})\to
\mathrm{Hom}_{A\text{-}A}(A\otimes_\Bbbk A,A).
\end{equation}
Up to a radical automorphism, $\mathbf{C}_{\mathcal{L}}$ maps the
$\alpha$ to the surjective multiplication morphism 
$\mathbf{m}:A\otimes_\Bbbk A\to A$. Dually,
up to a radical automorphism $\overline{\alpha}$ is 
mapped to the injective morphism $\mathbf{n}:A\to A\otimes_\Bbbk A$
given by the usual comultiplication on the Frobenius algebra $A$.
Hence the composition $\alpha\circ \overline{\alpha}$ 
is mapped to the endomorphism of $A$ given by multiplication 
with $\dim(A)s$, where $s$ is some fixed generator of the 
simple socle of $A$. In particular, this map is
non-zero if and only if $\dim(A)$ is not divisible by 
$\mathrm{char}(\Bbbk)$.

The $2$-morphism $(\alpha\circ \overline{\alpha})\circ_h \mathrm{id}_F$ thus gets
mapped to the endomorphism of $A\otimes_\Bbbk A$
sending $1\otimes 1$ to $s\otimes 1$. Similarly,
the $2$-morphism $\mathrm{id}_F\circ_h (\alpha\circ \overline{\alpha})$ thus gets
mapped to the endomorphism of $A\otimes_\Bbbk A$
sending $1\otimes 1$ to $1\otimes s$.

Denote by $Q$ the algebra ${\cC}(\mathrm{G},\mathrm{G})$. 
The representation map $\mathbf{C}_{\mathcal{L}}$
is an injective algebra morphism from $Q$
to $A\otimes_\Bbbk A^{\mathrm{op}}$, so we identify
$Q$ with its image in $A\otimes_\Bbbk A^{\mathrm{op}}$.

Combining the left and the right actions of ${\cC}$
on $\mathrm{add}(\mathrm{G})$ with the assumption that 
${\cC}$ is $\mathcal{J}$-simple, we can view this 
action as a cell $2$-representation of the fiat $2$-category
${\cC}\boxtimes_\Bbbk{\cC}^{\mathrm{co},\mathrm{op}}$.
In particular, $Q$ is the endomorphism algebra of a projective
object of this $2$-representation and hence is self-injective
(by \cite[Theorem~2]{KMMZ} and its dual version).

Evaluation at $L_\mathrm{G}$ defines a surjective algebra morphism from
$Q$ to $A$. Consequently, for any $x\in A$, there is  
$u_x\in A\otimes_\Bbbk R$ such that $Q$ contains
$(x\otimes 1) + u_x$. A similar argument for the right 
cell $2$-representation of ${\cC}$ implies that 
for any $x\in A$, there is  
$v_x\in A\otimes_\Bbbk R$ such that $Q$ contains
$(1\otimes x) + v_x$.

Now we see that all assumptions of Proposition~\ref{mainprop}
are satisfied. So, from Proposition~\ref{mainprop},
we have $Q=A\otimes_\Bbbk A^{\mathrm{op}}$.
Now, pre-composing the surjective map $\mathbf{m}$
with all possible endomorphism of $A\otimes_\Bbbk A^{\mathrm{op}}$
we obtain all possible $A$-$A$-homomorphisms 
from $A\otimes_\Bbbk A$ to $A$. This gives
the surjectivity of \eqref{eq-ad1} and completes the proof.
\end{proof}

\vspace{1cm}

\noindent
Volodymyr Mazorchuk, Department of Mathematics, Uppsala University,
Box 480, 751 06, Uppsala, SWEDEN, {\tt mazor\symbol{64}math.uu.se};
http://www.math.uu.se/$\tilde{\hspace{1mm}}$mazor/.
\vspace{0.1cm}

\noindent
Vanessa Miemietz, School of Mathematics, University of East Anglia,\\
Norwich NR4 7TJ, UK, \\ {\tt v.miemietz\symbol{64}uea.ac.uk};
http://www.uea.ac.uk/$\tilde{\hspace{1mm}}$byr09xgu/.

\end{document}